\documentclass[12pt,reqno]{amsart}
\usepackage{amsmath, amsfonts,epsfig, amssymb,amsbsy,eucal,mathrsfs,float, amsthm,mathtools,amsrefs}

\usepackage{graphicx,fullpage,xcolor}

\usepackage{tikz}

\usepackage[all]{xy}

\usepackage{tikz-cd}
\usepackage{tikz}

\usepackage{hyperref}
\hypersetup{
colorlinks = true,
urlcolor = blue,
citecolor = .
}

\usepackage[normalem]{ulem}

\usepackage{todonotes}

\numberwithin{equation}{section}
\theoremstyle{plain}

\newtheorem{theorem}{Theorem}[section]
\newtheorem{lemma}[theorem]{Lemma}
\newtheorem{proposition}[theorem]{Proposition}
\newtheorem{corollary}[theorem]{Corollary}
\newtheorem{conjecture}[theorem]{Conjecture}

\theoremstyle{definition}

\newtheorem{example}[theorem]{Example}

\theoremstyle{remark}

\newtheorem{remark}[theorem]{Remark}

\DeclareMathOperator{\QS}{QS}
\newcommand{\QSo}{\QS^\circ}
\newcommand{\QSx}{\QS^\times}

\newcommand{\adj}{\mathrm{ad}}
\newcommand{\coadj}{\mathrm{coad}}

\DeclareMathOperator{\IG}{IG}
\DeclareMathOperator{\OG}{OG}
\DeclareMathOperator{\Fl}{Fl}
\DeclareMathOperator{\Sp}{Sp}
\DeclareMathOperator{\QH}{QH}
\DeclareMathOperator{\QHc}{QH_{can}}
\DeclareMathOperator{\BQH}{BQH}
\DeclareMathOperator{\Ext}{Ext}
\DeclareMathOperator{\Hom}{Hom}
\DeclareMathOperator{\Aut}{Aut}

\DeclareMathOperator{\Spec}{Spec}

\DeclareMathOperator{\Rep}{Rep}
\DeclareMathOperator{\SO}{SO}
\DeclareMathOperator{\PSO}{PSO}
\DeclareMathOperator{\Spin}{Spin}
\DeclareMathOperator{\Pic}{Pic}
\DeclareMathOperator{\Conv}{Conv}

\DeclareMathOperator{\rk}{rk}

\newcommand{\bfP}{\mathbf{P}}
\newcommand{\E}{\mathbf{E}}
\newcommand{\bfS}{\mathbf{S}}

\DeclareMathOperator{\SL}{SL}
\DeclareMathOperator{\GL}{GL}
\DeclareMathOperator{\PGL}{PGL}
\DeclareMathOperator{\tGL}{{\widetilde{GL}}}

\newcommand{\bA}{{\mathbb A}}
\newcommand{\bC}{{\mathbb C}}

\newcommand{\bL}{{\mathbb L}}
\newcommand{\bR}{{\mathbb R}}

\newcommand{\bS}{{\mathbb S}}

\newcommand{\bP}{{\mathbb P}}
\newcommand{\bQ}{{\mathbb Q}}
\newcommand{\bZ}{{\mathbb Z}}

\newcommand{\rA}{\mathrm{A}}
\newcommand{\rB}{\mathrm{B}}
\newcommand{\rC}{\mathrm{C}}
\newcommand{\rD}{\mathrm{D}}
\newcommand{\rE}{\mathrm{E}}
\newcommand{\rF}{\mathrm{F}}
\newcommand{\rG}{\mathrm{G}}
\newcommand{\rH}{\mathrm{H}}
\newcommand{\rK}{\mathrm{K}}
\newcommand{\rL}{\mathrm{L}}
\newcommand{\rP}{\mathrm{P}}
\newcommand{\rT}{\mathrm{T}}
\newcommand{\rTs}{\mathrm{T}_{\mathrm{short}}}
\newcommand{\rV}{\mathrm{V}}
\newcommand{\rW}{\mathrm{W}}

\newcommand{\rs}{\mathrm{s}}

\newcommand{\cA}{\mathcal{A}}
\newcommand{\cB}{\mathcal{B}}
\newcommand{\cC}{\mathcal{C}}

\newcommand{\cE}{\mathcal{E}}
\newcommand{\cO}{\mathcal{O}}
\newcommand{\cR}{\mathcal{R}}
\newcommand{\cS}{\mathcal{S}}
\newcommand{\cU}{\mathcal{U}}
\newcommand{\cV}{\mathcal{V}}

\newcommand{\cF}{\mathcal{F}}

\newcommand{\fS}{\mathfrak{S}}

\newcommand{\Db}{{\mathbf D^{\mathrm{b}}}}

\begin{document}

\title{Residual categories for (co)adjoint Grassmannians\\in classical types}

\author{Alexander Kuznetsov}

\address{
\parbox{0.95\textwidth}{
Algebraic Geometry Section, Steklov Mathematical Institute of Russian Academy of Sciences,
8 Gubkin str., Moscow 119991 Russia
\smallskip
}
}


\email{akuznet@mi-ras.ru}

\author{Maxim Smirnov}
\address{
\parbox{0.95\textwidth}{
Universit\"at Augsburg,
Institut f\"ur Mathematik,
Universit\"atsstr.~14,
86159 Augsburg,
Germany
\smallskip
}}
\email{maxim.smirnov@math.uni-augsburg.de}

\thanks{This work is supported by the Russian Science Foundation under grant~19-11-00164.}

\maketitle

\begin{abstract}
In our previous paper we suggested a conjecture relating the structure of the small quantum cohomology ring
of a smooth Fano variety of Picard number~1 to the structure of its derived category of coherent sheaves.
Here we generalize this conjecture, make it more precise, and support by the examples of (co)adjoint homogeneous varieties
of simple algebraic groups of Dynkin types~$\rA_n$ and~$\rD_n$,
i.e., flag varieties~$\Fl(1,n;n+1)$ and isotropic orthogonal Grassmannians~$\OG(2,2n)$;
in particular we construct on each of those an exceptional collection invariant with respect to the entire automorphism group.

For $\OG(2,2n)$ this is the first exceptional collection proved to be full.
\end{abstract}

\section{Introduction}

This paper is devoted to the study of derived categories of coherent sheaves on homogeneous varieties of semisimple algebraic groups
and the relation to their quantum cohomology.

Recall that Dubrovin's conjecture (see~\cite{Du}) predicts,
that the existence of a full exceptional collection in the bounded derived category $\Db(X)$ of coherent sheaves
on a smooth projective variety $X$ is equivalent to the generic semisimplicity of its big quantum cohomology ring~$\BQH(X)$
(for background on quantum cohomology we refer to~\cites{Manin, FP}).
The big quantum cohomology ring is usually very hard to compute, in contrast to the small quantum cohomology $\QH(X)$.
Of course, if $\QH(X)$ is generically semisimple, then so is $\BQH(X)$.
Thus, from Dubrovin's conjecture we conclude that generic semisimplicity of~$\QH(X)$ should imply
the existence of a full exceptional collection in~$\Db(X)$.
On the other hand, it was observed that the opposite implication is not true, see~\cites{CMKMPS,GMS,Ke,Pe}.

This observation suggests that on the one hand, some mildly non-simple factors of the general fiber of $\QH(X)$
should not obstruct the existence of a full exceptional collection in $\Db(X)$,
and on the other hand, generic semisimplicity of $\QH(X)$ should have stronger implications for~$\Db(X)$
than just the existence of a full exceptional collection.

In~\cite{KS2020} we stated for varieties of Picard number~1 a conjecture saying that
the structure of the general fiber of $\QH(X)$ determines the structure of the exceptional collection on~$X$.
In this paper we suggest a more general and precise version of this conjecture and support it by new examples.

To state the conjecture we will need some notation.
First, recall that the {\sf index} of a smooth projective variety~$X$
is the maximal integer~$m$ such that the canonical class~$K_X$ is divisible by~$m$ in~$\Pic(X)$;
we usually assume that~$X$ is a Fano variety over an algebraically closed field of characteristic zero
and write
\begin{equation*}
\omega_X \cong \cO_X(-m),
\end{equation*}
where $\cO_X(1)$ is a primitive ample line bundle on~$X$.

We will say that an exceptional collection $E_1,\dots,E_k$ in $\Db(X)$ {\sf extends to a rectangular Lefschetz collection},
if the collection
\begin{equation}
\label{eq:rectangular-collection}
E_1, E_2, \dots, E_k;\
E_1(1), E_2(1), \dots, E_k(1);\ \dots;\
E_1(m-1), E_2(m-1), \dots, E_k(m-1)
\end{equation}
is also exceptional.
The orthogonal complement
\begin{equation}
\label{eq:residual-category}
\cR =
\Big\langle
E_1,
\dots, E_k;\
\dots;\
E_1(m-1),
\dots, E_k(m-1)
\Big\rangle^\perp
\subset \Db(X)
\end{equation}
is called the {\sf residual category} of the above collection.
In~\cite[Theorem~2.8]{KS2020} we checked that~$\cR$ is endowed with an autoequivalence $\tau_\cR \colon \cR \to \cR$
(called the {\sf induced polarization} of~$\cR$) such that
\begin{equation*}
\tau_\cR^m \cong \bS_\cR^{-1}[\dim X],
\end{equation*}
where $\bS_\cR$ is the Serre functor of~$\cR$.
Thus, $\tau_\cR$ plays in~$\cR$ the same role as the twist by~$\cO(1)$ plays in~$\Db(X)$.

Assume $X$ is a smooth Fano variety and let $r$ be the Picard rank of~$X$, so that $\QH(X)$ is an algebra
over the ring $\bQ[q_1,\dots,q_r]$ of functions on the affine space $\Pic(X) \otimes \bQ$ over~$\bQ$.
Let
\begin{equation*}
\QHc(X) := \QH(X) \otimes_{\bQ[q_1,\dots,q_r]} \bC
\end{equation*}
be the base change of $\QH(X)$ to the point of $\Spec(\bQ[q_1,\dots,q_r])$ corresponding to the canonical class of~$X$;
this is a $\bC$-algebra, whose underlying vector space is canonically isomorphic to $\rH^\bullet(X,\bC)$.
Assume that $\rH^{\mathrm{odd}}(X,\bC) = 0$ and let $m$ be the index of~$X$.
Then $\QHc(X)$ is commutative, so we consider the finite scheme
\begin{equation*}
\QS_X := \Spec(\QHc(X))
\end{equation*}
and call it the {\sf (canonical) quantum spectrum} of~$X$.
By Dimension Axiom for GW invariants the natural grading of $\rH^\bullet(X,\bC)$ induces a~$\bZ/m$-grading of $\QHc(X)$ such that
\begin{equation*}
\deg(\rH^{2i}(X,\bQ)) \equiv i \pmod m,
\end{equation*}
which gives rise to an action of the group $\mu_m$ on the quantum spectrum~$\QS_X$.

Furthermore, let
\begin{equation*}
-\rK_X \in \rH^2(X,\bC) \subset \QHc(X)
\end{equation*}
be the anticanonical class.
It defines a morphism of algebras $\bC[\kappa] \to \QHc(X)$, $\kappa \mapsto -\rK_X$ from a polynomial algebra in the variable~$\kappa$,
which geometrically can be understood as a morphism of schemes
\begin{equation}
\label{eq:kappa}
\kappa \colon \QS_X \to \bA^1,
\end{equation}
which is $\mu_m$-equivariant for the action on $\QS_X$ defined above and the standard action on~$\bA^1$.

\begin{remark}
The notion of quantum spectrum is parallel to that of spectral cover in the theory of Frobenius manifolds (see~\cite{Manin,Hertling}).
The map \eqref{eq:kappa} is analogous to the restriction of a Landau--Ginzburg potential to its critical locus.
\end{remark}

\begin{example}
Let $X = \bP^n$, so that~$r = 1$ and $m = n + 1$.
Then
\begin{equation*}
\QH(X) \cong \bQ[h,q]/(h^{n+1}-q),
\qquad \text{hence} \qquad
\QHc(X) \cong \bC[h]/(h^{n+1} - 1).
\end{equation*}
The quantum spectrum~$\QS_{\bP^n}$ is the reduced subscheme of~$\bA^1$ with points~$\zeta^i$, $0 \le i \le n$,
where~$\zeta$ is a primitive $(n+1)$-st root of unity.
The $\bZ/(n+1)$-grading is defined by $\deg(h) = 1$, it induces the natural action of~$\mu_{n+1}$ on~$\bA^1$
under which~$\QS_{\bP^n}$ is invariant.
Finally, we have~$-\rK_X = (n+1)h$, and the map~$\kappa$ up to rescaling is the natural inclusion $\QS_{\bP^n} \hookrightarrow \bA^1$.
\end{example}

Our conjecture is based on an analogy between the~$\mu_m$-action on~$\QS_X$ and the twist by~$\cO_X(1)$
on Lefschetz exceptional collections in~$\Db(X)$:
in this analogy the parts of~$\QS_X$ supported over
the complement of the origin and the origin of $\bA^1$ (with respect to the map~$\kappa$)
\begin{equation*}
\QSx_X := \kappa^{-1}(\bA^1 \setminus \{0\}),
\qquad
\QSo_X := \QS_X \setminus \QSx_X.
\end{equation*}
correspond to the rectangular part and the residual category of the Lefschetz collection.
Note that the $\mu_m$-action on~$\QSx_X$ is free (because it is free on~$\bA^1 \setminus \{0\}$),
hence there exists a finite subscheme $Z \subset \QSx_X$ such that the action map $\mu_m \times Z \to \QSx_X$ is an isomorphism;
the length of~$Z$ is equal to the length of~$\QSx_X$ divided by~$m$,
and we consider it as an analogue of the subcollection~$E_1,\dots,E_k$ in~\eqref{eq:rectangular-collection}.

Note that generic semisimplicity of~$\BQH(X)$ assumed in the following conjecture
implies the vanishing of~$\rH^{\mathrm{odd}}(X,\bC)$ assumed above by~\cite[Theorem 1.3]{HeMaTe}.
\begin{conjecture}
\label{conjecture:main}
Let $X$ be a Fano variety of index~$m$ over an algebraically closed field of characteristic zero and assume that the big quantum cohomology~$\BQH(X)$ is generically semisimple.
\begin{enumerate}
\item
There is an $\Aut(X)$-invariant exceptional collection $E_1,\dots,E_k$ in $\Db(X)$, where~$k$
is the length of~$\QSx_X$ divided by~$m$;
this collection extends to a rectangular Lefschetz collection~\eqref{eq:rectangular-collection} in $\Db(X)$.
\item
The residual category $\cR$ of this collection \textup(defined by~\eqref{eq:residual-category}\textup)
has a completely orthogonal $\Aut(X)$-in\-va\-ri\-ant
decomposition
\begin{equation*}
\cR = \bigoplus_{\xi \in \QSo_X} \cR_\xi
\end{equation*}
with components indexed by closed points $\xi \in \QSo_X$;
moreover, the component~$\cR_\xi$ of~$\cR$ is generated by an exceptional collection
of length equal to the length of the localization~$(\QSo_X)_\xi$ at~$\xi$.
\item
The induced polarization $\tau_\cR$ permutes the components $\cR_\xi$; more precisely,
for each point $\xi \in \QSo_X$ it induces an equivalence
\begin{equation*}
\tau_\cR \colon \cR_\xi \xrightarrow{\ \sim\ } \cR_{g(\xi)},
\end{equation*}
where $g$ is a generator of $\mu_m$.
\end{enumerate}
\end{conjecture}

Note that any exceptional object on~$X$ is invariant with respect to any \emph{connected} reductive group acting on~$X$
(see~\cite[Lemma~2.2]{Po}), however, the group $\Aut(X)$ is not connected in general.
Thus, $\Aut(X)$-invariance of the collection is an extra constraint (cf.\ the discussions in~\S\ref{subsection:theorem-an}
and~\S\ref{subsection:theorem-dn}).

Considerations from the introduction of~\cite{KS2020}
provide a Homological Mirror Symmetry justification for Conjecture~\ref{conjecture:main}
(except possibly for the~$\Aut(X)$-invariance statements, and for the action of~$\tau_\cR$ on the residual category).

If the Picard rank of~$X$ is~1, so that $\QH(X)$ is a $\bQ[q]$-algebra, we have
\begin{equation*}
\QH(X) \otimes_{\bQ[q]} \bQ(q) \cong (\QH(X) \otimes_{\bQ[q]} \bQ) \otimes_\bQ \bQ(q)
\end{equation*}
because $\QH(X)$ is $\bZ$-graded with $\deg(q) = m$, hence semisimplicity of $\QHc(X)$ is equivalent to generic semisimplicity of~$\QH(X)$.
Moreover, if $\QH(X)$ is generically semisimple the length of each localization $(\QSo_X)_\xi$ is~1,
thus Conjecture~\ref{conjecture:main} predicts the existence in~$\Db(X)$ of a rectangular Lefschetz collection
whose residual category is generated by a completely orthogonal exceptional collection,
which is equivalent to the prediction of~\cite[Conjecture~1.12]{KS2020}.
In particular, numerous examples listed in the introduction to~\cite{KS2020}
together with the main result of~\cite{KS2020} support both conjectures.

Let us also discuss a couple of simple examples of varieties with higher Picard rank.

\begin{example}
\label{ex:pn-pn}
Let $X = \bP^n \times \bP^n$, so that~$r = 2$ and~$m = n + 1$.
By the quantum K\"unneth formula (see \cite{Kaufmann, KMK}) we have
\begin{equation*}
\QH(X) \cong \QH(\bP^n) \otimes_\bQ \QH(\bP^n) \cong \bQ[h_1,h_2,q_1,q_2]/(h_1^{n+1} - q_1,\ h_2^{n+1} - q_2).
\end{equation*}
The canonical class direction corresponds to $q_1 = q_2$, so the canonical quantum cohomology ring
can be computed by specializing both~$q_1$ and~$q_2$ to~1:
\begin{equation*}
\QHc(X) \cong \bC[h_1,h_2]/(h_1^{n+1} - 1,\ h_2^{n+1} - 1).
\end{equation*}
Its spectrum $\QS_{X}$ is a reduced scheme of length $(n+1)^2$ with points $(\zeta^i,\zeta^j)$, $0 \le i,j \le n$,
where~$\zeta$ is the primitive $(n+1)$-st root of unity.
The function~$\kappa$ (up to rescaling) is given by~$\kappa = h_1 + h_2$, so $\kappa(\zeta^i,\zeta^j) = \zeta^i + \zeta^j$
and
\begin{equation}
\label{eq:quantum-pn-pn}
\QSo_{X} = {}
\begin{cases}
\varnothing, & \text{if $n = 2k$ is even,}\\
\{ (\zeta^i,\zeta^{k+1+i}) \}, & \text{if $n = 2k + 1$ is odd.}
\end{cases}
\end{equation}
Furthermore, the generator of $\mu_{n+1}$ acts by~$(\zeta^i,\zeta^j) \mapsto (\zeta^{i+1},\zeta^{j+1})$;
in particular the action of~$\mu_{n+1}$ on $\QSo_{X}$ is simply transitive.

The formula~\eqref{eq:quantum-pn-pn} exhibits a difference between the case of even and odd~$n$;
it also appears on the level of derived category.
If~$n = 2k$, the collection of $2k + 1$ line bundles
\begin{equation}
\label{eq:pnpn-block0}
\cA = \langle \cO, \cO(1,0), \cO(0,1), \dots, \cO(k,0), \cO(0,k) \rangle
\end{equation}
extends to an $\Aut(X)$-invariant rectangular Lefschetz collection~$\Db(X) = \langle \cA,\dots,\cA(n)\rangle$
of total length $(n+1)(2k+1) = (n+1)^2$,
whose residual category is zero (this can be easily proved by the argument of Lemma~\ref{lemma:an-even}).

If~$n = 2k + 1$, the collection~\eqref{eq:pnpn-block0} still
extends to an $\Aut(X)$-invariant rectangular Lefschetz collection in~$\Db(X)$, this time of length~$(n+1)n$,
and it can be checked (a similar computation in a more complicated situation can be found in~\S\ref{section:an})
that its residual category is generated by~\mbox{$n + 1 = 2k + 2$} completely orthogonal exceptional vector bundles
\begin{equation*}
\begin{aligned}
F_1 &= \cO(-1,k),\\
F_2 &= \tau_\cR(F_1) 		&&\cong \hphantom{(1)}\cO \boxtimes \Omega^{k+1}(k+1),\\
F_3 &= \tau_\cR(F_2) 		&&\cong \cO(1) \boxtimes \Omega^{k+2}(k+2),\\
&\vdots \\
F_{k+2} &= \tau_\cR(F_{k+1}) 	&&\cong \cO(k) \boxtimes \Omega^{2k+1}(2k+1)	&&\cong \cO(k,-1),\\
F_{k+3} &= \tau_\cR(F_{k+2}) 		&&\cong \Omega^{k+1}(k+1) \boxtimes \cO,\\
F_{k+4} &= \tau_\cR(F_{k+3}) 		&&\cong \Omega^{k+2}(k+2) \boxtimes \cO(1),\\
&\vdots \\
F_{2k+2} &= \tau_\cR(F_{2k+1}) 	&&\cong \Omega^{2k}(2k) \boxtimes \cO(k-1)
\end{aligned}
\end{equation*}
(where isomorphisms are up to shift).
Note that the action of $\tau_\cR$ on this exceptional collection is simply transitive,
analogously to the action of $\mu_{n+1}$ on~$\QSo_{X}$.
\end{example}

\begin{example}
Let $X = (\bP^1)^n$, so that~$r = n$ and~$m = 2$.
Applying again the quantum K\"unneth formula one can check that
\begin{equation*}
\QS_{(\bP^1)^n} = {}
\{ (\pm1, \pm1, \dots, \pm1) \} \subset \bA^n
\end{equation*}
is a reduced scheme of length~$2^n$
and the function $\kappa$ is given by the sum of coordinates.
Therefore, $\QSo_{(\bP^1)^n}$ is empty when $n$ is odd, while for even $n = 2k$ it contains exactly $\binom{2k}{k}$ points,
and the $\mu_2$-action splits this set into $\tfrac12\binom{2k}{k}$ free orbits.

On the level of derived categories the same thing happens.
If $n$ is odd, a rectangular~$\Aut((\bP^1)^n)$-invariant Lefschetz collection in~$\Db((\bP^1)^n)$
with zero residual category was constructed in~\cite[Theorem~4.1]{Mironov}.
If $n = 2k$ is even, using~\cite[Theorem~4.1]{Mironov} it is easy to show that the residual category
is generated by $\binom{2k}{k}$ exceptional line bundles and the $\tau_\cR$-action swaps them (up to shift) pairwise.
\end{example}

Using other results from~\cite{Mironov} one can verify Conjecture~\ref{conjecture:main} for some other products~$(\bP^n)^k$.
In all these examples, however, the ring~$\QHc(X)$ is semisimple.
Below we discuss more intricate examples with non-semisimple ring~$\QHc(X)$,
provided by homogeneous varieties of simple algebraic groups,
where quite a lot is known both about quantum cohomology and derived categories.

Perhaps, the most interesting case here is that of \emph{adjoint} and \emph{coadjoint} homogeneous varieties.
Recall that an {\sf adjoint} (resp.\ {\sf coadjoint}) homogeneous variety of a simple algebraic group $\rG$
is the highest weight vector orbit in the projectivization of the irreducible $\rG$-representation,
whose highest weight is the highest \emph{long} (resp.\ \emph{short}) root of~$\rG$;
in particular, if the group~$\rG$ is simply laced, the adjoint and coadjoint varieties coincide.

For classical Dynkin types adjoint and coadjoint varieties are:
\begin{equation*}
\begin{array}{|c|c|c|c|}
\hline
\text{Dynkin type} &
\text{group~$\rG$} &
\text{adjoint variety} &
\text{coadjoint variety}
\\
\hline
\rA_n & \SL(n+1) & \Fl(1,n;n+1) & \Fl(1,n;n+1)
\\
\hline
\rB_n & \Spin(2n+1) & \OG(2,2n+1) & Q^{2n-1}
\\
\hline
\rC_n & \Sp(2n) & \bP^{2n-1} & \IG(2,2n)
\\
\hline
\rD_n & \Spin(2n) & \OG(2,2n) & \OG(2,2n)
\\
\hline
\end{array}
\end{equation*}
Here $\Fl(1,n;n+1)$ is the partial flag variety, $Q^k$ is a (smooth) $k$-dimensional quadric,
while~$\IG(2,2n)$, $\OG(2,2n)$, and~$\OG(2,2n+1)$ are the symplectic and orthogonal isotropic Grassmannians of 2-dimensional subspaces, respectively.
Note that the Picard rank of a (co)adjoint variety is~1, except for the $\rA_n$-case, where it is~2.

The small quantum cohomology ring of (co)adjoint varieties was computed in~\cite{BKT09,ChPe,CF99,Kim} in terms of generators and relations.
Using these results, the fiber~$\QSo_X$ of the map~$\kappa$ defined in~\eqref{eq:kappa}
was computed by Nicolas Perrin and the second named author in~\cite{PeSm}.
To state the results of~\cite{PeSm} we will need some notation.

Let $\rT(\rG)$ be the Dynkin diagram of~$\rG$, and let $\rTs(\rG)$ be the subdiagram of $\rT(\rG)$
with vertices corresponding to \emph{short} roots.
For reader's convenience we collect the resulting Dynkin types in a table:
\begin{equation*}
\begin{array}{|c|c|c|c|c|c|c|c|}
\hline
\rT & \rA_n & \rB_n & \rC_n & \rD_n & \rE_n & \rF_4 & \rG_2 \\
\hline
\rTs & \rA_n & \rA_1 & \rA_{n-1} & \rD_n & \rE_n & \rA_2 & \rA_1 \\
\hline
\end{array}
\end{equation*}
The following theorem describes~$\QS_X^\circ$ for adjoint and coadjoint varieties.

\begin{theorem}[{\cite{PeSm}}]
\label{theorem:coadjoint}
Let $X^\adj$ and $X^\coadj$ be the adjoint and coadjoint varieties of a simple algebraic group~$\rG$, respectively.
\begin{enumerate}
\item
If $\rT(\rG) = \rA_{2n}$, then $\QSo_{X^\adj} = \QSo_{X^\coadj} = \varnothing$.
\item
If $\rT(\rG) \ne \rA_{2n}$, then $\QSo_{X^\coadj}$ is a single non-reduced point and the localization of~$\QHc(X^\coadj)$ at this point
is isomorphic to the Jacobian ring of a simple hypersurface singularity of type~$\rTs(\rG)$.
\item
If $\rT(\rG)$ is simply laced, then we have $X^\adj = X^\coadj$ and~$\QSo_{X^\adj} = \QSo_{X^\coadj}$.
\item
If $\rT(\rG)$ is not simply laced, then $\QSo_{X^\adj} = \varnothing$.
\end{enumerate}
\end{theorem}

A combination of Theorem~\ref{theorem:coadjoint} with Conjecture~\ref{conjecture:main}
allows us to make predictions about the structure of derived categories of adjoint and coadjoint varieties.
Our expectation is stated in the following two conjectures.

\begin{conjecture}
\label{conjecture:adjoint}
Let $X$ be the adjoint variety of a simple algebraic group~$\rG$ over an algebraically closed field of characteristic zero.
If $\rT(\rG)$ is not simply laced, then $\Db(X)$ has a full $\Aut(X)$-invariant rectangular Lefschetz exceptional collection.
\end{conjecture}

\begin{conjecture}
\label{conjecture:coadjoint}
Let $X$ be the coadjoint variety of a simple algebraic group~$\rG$ over an algebraically closed field of characteristic zero.
Then $\Db(X)$ has an~$\Aut(X)$-invariant rectangular Lefschetz exceptional collection with residual category~$\cR$ and
\begin{enumerate}
\item
if $\rT(\rG) = \rA_n$ and $n$ is even, then $\cR = 0$;
\item
otherwise, $\cR$ is equivalent to the derived category of representations of a quiver of Dynkin type~$\rTs(\rG)$.
\end{enumerate}
\end{conjecture}

In non-simply laced Dynkin types~$\rB_n$, $\rC_n$, $\rF_4$, and~$\rG_2$ these expectations agree with known results about~$\Db(X)$.
Indeed, for adjoint varieties full rectangular Lefschetz decompositions were constructed in~\cite[Theorem~7.1]{Ku08a} for type~$\rB_n$,
\cite[Example~1.4]{KS2020} for type~$\rC_n$, and~\cite[\S6.4]{K06} for type~$\rG_2$.
For coadjoint varieties the residual categories were computed in~\cite[Theorem~9.6]{CMKMPS} for type~$\rC_n$,
in~\cite[Example~1.6]{KS2020} for types $\rB_n$ and~$\rG_2$,
and in~\cite[Theorem~1.4]{BKS} for type~$\rF_4$.
So, the only non-simply laced case that is still not known is that of the adjoint variety of type~$\rF_4$.

The main result of this paper is the proof of Conjecture~\ref{conjecture:coadjoint} for Dynkin types $\rA_n$ and $\rD_n$.
Since these Dynkin types are simply laced, the coadjoint and adjoint varieties coincide.

\begin{theorem}
\label{theorem:intro}
Let~$X$ be the \textup(co\textup)adjoint variety of a simple algebraic group~$\rG$ of Dynkin type~$\rA_n$ or~$\rD_n$
over an algebraically closed field of characteristic zero.
Then~$\Db(X)$ has an~$\Aut(X)$-invariant rectangular Lefschetz exceptional collection with residual category~$\cR$ and
\begin{enumerate}
\item
\label{theorem:intro:an-even}
if $\rT(\rG) = \rA_n$ and $n$ is even, then $\cR = 0$;
\item
\label{theorem:intro:an-odd}
if $\rT(\rG) = \rA_n$ and $n$ is odd, then~$\cR \cong \Db(\rA_n)$;
\item
\label{theorem:intro:dn}
if $\rT(\rG) = \rD_n$,
then~$\cR \cong \Db(\rD_n)$;
\end{enumerate}
where $\Db(\rA_n)$ and $\Db(\rD_n)$ are the derived categories of representations
of quivers of Dynkin types~$\rA_n$ and~$\rD_n$, respectively.
\end{theorem}

More precise versions of these results can be found in Theorem~\ref{theorem:an} (for type~$\rA_n$)
and Theorem~\ref{theorem:dn} (for type~$\rD_n$) in the body of the paper.
We leave the remaining exceptional types $\rE_6$, $\rE_7$, $\rE_8$, and~$\rF_4$ for future work.

Note that a part of the statement of this theorem is a construction of a full exceptional collection in~$\Db(\OG(2,2n))$,
which was not known before (see~\cite{KuPo} for a survey of results about exceptional collections on homogeneous varieties).

\begin{remark}
As one can see from Theorem~\ref{theorem:coadjoint} and Conjecture~\ref{conjecture:coadjoint},
the case of Dynkin type~$\rA_n$ with even~$n$ is somewhat special.
In this case, the Picard rank is equal to~2, and the canonical quantum cohomology ring $\QHc(X)$ is semisimple,
so a singularity of type~$\rTs(\rG) = \rA_n$ does not show up.
However, one can see this singularity in the \emph{skew-canonical} quantum cohomology ring,
i.e., the ring obtained from $\QH(X)$ by base change to the point of~$\Pic(X) \otimes \bQ$ corresponding to the line bundle $\cO(1,-1)$. It would be very interesting to find a categorical interpretation of this fact.
\end{remark}

\medskip

\noindent {\bf Acknowledgements.}
We are indebted to Nicolas Perrin for sharing with us his results of quantum cohomology computations and attracting our attention
to the coadjoint varieties of types $\rA_n$ and $\rD_n$, that eventually led to this paper.
We thank Giordano Cotti, Anton Fonarev, Sergey Galkin, and Anton Mellit for useful discussions
and the anonymous referee for their comments.
Further, we are very grateful to Pieter Belmans for his kind permission to reuse some parts of the code written for \cite{BS},
which was instrumental for this paper, and for his comments on the first draft of this paper.
Finally, M.S. would like to thank ICTP in Trieste, and MPIM in Bonn, where a part of this work was accomplished, for their hospitality.

\section{Type $\rA_n$}
\label{section:an}

In this section we prove parts~\eqref{theorem:intro:an-even} and~\eqref{theorem:intro:an-odd} of Theorem~\ref{theorem:intro},
restated in a more precise form in Theorem~\ref{theorem:an} below.
In this section we work over an arbitrary field~$\Bbbk$.

\subsection{Statement of the theorem}
\label{subsection:theorem-an}

Let $V$ be a vector space of dimension~$n + 1$.
Throughout this section we put
\begin{equation*}
X = \Fl(1,n;n+1) = \Fl(1,n;V) \subset \bP(V) \times \bP(V^\vee);
\end{equation*}
note that in this embedding $X$ is a hypersurface of bidegree $(1,1)$.

The automorphism group of $X$ is the semidirect product
\begin{equation*}
\Aut(X) \cong \PGL(V) \rtimes \bZ/2,
\end{equation*}
where the factor $\bZ/2$ acts by an outer automorphism (corresponding to the symmetry of the Dynkin diagram $\rA_n$)
that is induced by the morphism
\begin{equation*}
\sigma_B \colon \bP(V) \times \bP(V^\vee) \xrightarrow{\ B \times B^{-1}\ } \bP(V^\vee) \times \bP(V)
\end{equation*}
given by a choice of non-degenerate bilinear form $B$ on $V$.

It is elementary to construct a rectangular Lefschetz decomposition for $\Db(X)$
by using the $\bP^{n-1}$-fibration structure $X \to \bP(V)$ of $X$
(see the proofs of Lemma~\ref{lemma:an-even} and Lemma~\ref{lemma:an-standard-ec}).
This Lefschetz decomposition is automatically $\PGL(V)$-invariant and its residual category is trivial.
However, it is \emph{not} $\Aut(X)$-invariant, since the outer automorphism
takes it to a decomposition associated with the other $\bP^{n-1}$-fibration $X \to \bP(V^\vee)$,
and so does not preserve the original one.

In this section we construct in $\Db(X)$ a rectangular $\Aut(X)$-invariant Lefschetz collection
and compute its residual category.
To state the result we need some notation.
Let
\begin{equation*}
0 \hookrightarrow \cU_1 \hookrightarrow \cU_n \hookrightarrow V \otimes \cO
\end{equation*}
be the tautological flag of rank-1 and rank-$n$ subbundles in the trivial vector bundle.
We set
\begin{equation}
\label{def:ce}
\cE := \cU_n/\cU_1
\end{equation}
for the intermediate quotient.
Note that $\sigma_B^*\cE \cong \cE^\vee$ and $\det(\cE) \cong \cO(1,-1)$,
where we denote by~$\cO(a_1,a_2)$ the restriction to~$X$ of the line bundle $\cO(a_1) \boxtimes \cO(a_2)$ on $\bP(V) \times \bP(V^\vee)$.

We prove the following

\begin{theorem}
\label{theorem:an}
Set $k := \lfloor n/2 \rfloor$.
The collection of $2k + 1$ line bundles
\begin{equation}
\label{eq:ca-an}
\cA := \big\langle \cO(0,0), \cO(1,0), \cO(0,1), \cO(2,0), \cO(0,2), \dots, \cO(k,0), \cO(0,k) \big\rangle
\end{equation}
in $\Db(X)$ is exceptional and extends to an $\Aut(X)$-invariant semiorthogonal decomposition
\begin{equation}
\label{eq:rectangular-an}
\Db(X) = \langle \cR, \cA, \cA \otimes \cO(1,1), \dots, \cA \otimes \cO(n-1,n-1) \rangle
\end{equation}
where $\cR$ is the residual category.

$(1)$ If $n = 2k$ the residual category is zero.

$(2)$ If $n = 2k + 1$ the residual category is generated by the $\Aut(X)$-invariant exceptional collection
\begin{multline}
\label{eq:residual-final}
\cR = \langle \cO(-1,k), \cO(k,-1); \cE(-1,k-1), \cE^\vee(k-1,-1);
\\
\dots;
\Lambda^{k-1}\cE(-1,1), \Lambda^{k-1}\cE^\vee(1,-1);
\Lambda^k\cE(-1,0) \cong \Lambda^k\cE^\vee(0,-1) \rangle
\end{multline}
of length $n = 2k + 1$ and is equivalent to the derived category of the Dynkin quiver $\rA_n$.
\end{theorem}

In Remark~\ref{remark:an-weights} we provide a description of $\cA$ and $\cR$ in terms of weights of $\SL(n+1)$.

In Corollary~\ref{cor:rectangular-part-an} we check that the collection of line bundles
defining the rectangular part of~\eqref{eq:rectangular-an} is exceptional.
Furthermore, in Lemma~\ref{lemma:an-even} (for even $n$) and Lemma~\ref{lemma:mutation-1} (for odd~$n$)
we extend it to a full exceptional collection of line bundles.
In the construction we use the advantage of already knowing a full exceptional collection on~$X$
(thanks to the $\bP^{n-1}$-fibration mentioned above),
so it is enough to rearrange it appropriately by a sequence of mutations
(we refer to~\cite[\S2.3]{IK} for a summary of results we need; Lemma~2.13 of \emph{loc.\ cit.}\/ is especially useful).
Part~(1) of the theorem is proved in Lemma~\ref{lemma:an-even} and part~(2) in Lemma~\ref{lemma:residual-an}.

For verifications of exceptionality we will need the following

\begin{lemma}
\label{lemma:lb-vanishing}
If either of the following conditions
\begin{equation*}
1 - n \le a \le -1,
\quad\text{or}\quad
1 - n \le b \le -1,
\quad\text{or}\quad
(a,b) = (0,-n),
\quad\text{or}\quad
(a,b) = (-n,0)
\end{equation*}
is satisfied, then the cohomology $H^\bullet(X,\cO(a,b))$ vanishes.
\end{lemma}

\begin{proof}
Consider the standard exact sequence
\begin{equation*}
0 \to \cO_{\bP(V) \times \bP(V^\vee)}(a-1,b-1) \to \cO_{\bP(V) \times \bP(V^\vee)}(a,b) \to \cO(a,b) \to 0.
\end{equation*}
If any of the conditions listed in the lemma is satisfied, then
\begin{equation*}
H^\bullet(\bP(V) \times \bP(V^\vee), \cO_{\bP(V) \times \bP(V^\vee)}(a,b))
\quad\text{and}\quad
H^\bullet(\bP(V) \times \bP(V^\vee), \cO_{\bP(V) \times \bP(V^\vee)}(a-1,b-1))
\end{equation*}
vanish, hence $H^\bullet(X,\cO(a,b))$ vanishes as well.
\end{proof}

One of the consequences of this computation is the following

\begin{corollary}
\label{cor:rectangular-part-an}
The collection of objects in the right side of~\eqref{eq:ca-an} is an exceptional collection.
Moreover, the components~$\cA \otimes \cO(t,t)$ in~\eqref{eq:rectangular-an} are semiorthogonal for~$0 \le t \le n-1$.
\end{corollary}

\begin{proof}
To prove the corollary we need to show that
\begin{align*}
\Ext^\bullet(\cO(j+t,t),\cO(i,0)) &= H^\bullet(X,\cO(i-j-t,-t))\\
\intertext{and}
\Ext^\bullet(\cO(t,j+t),\cO(i,0)) &= H^\bullet(X,\cO(i-t,-j-t))
\end{align*}
both vanish when $0 \le i,j \le k$ and $1 \le t \le n-1$.

The first vanishing is clear since~$1-n \le -t \le -1$ and Lemma~\ref{lemma:lb-vanishing} applies.
The second vanishing also follows if $-j-t \ge 1-n$.
So, assume $-j-t \le -n$.
Then $t \ge n - j \ge n - k$, hence $i - t \le k - (n - k) = 2k - n \le 0$.
Since also $i - t \ge -t \ge 1 - n$, the vanishing also follows from Lemma~\ref{lemma:lb-vanishing}, unless $i - t = 0$.
But then we must have $i = j = t = k$ and~$n = 2k$, so that the corresponding line bundle is~$\cO(0,-n)$,
and its cohomology vanishes, again by Lemma~\ref{lemma:lb-vanishing}.
\end{proof}

We denote the natural projections of~$X$ by
\begin{equation*}
p_1 \colon X \to \bP(V)
\qquad\text{and}\qquad
p_2 \colon X \to \bP(V^\vee).
\end{equation*}
For any pair of coherent sheaves on $\bP(V)$ and $\bP(V^\vee)$ we set
\begin{equation*}
\cF_1 \boxtimes_X \cF_2 := p_1^*\cF_1 \otimes p_2^*\cF_2 \cong (\cF_1 \boxtimes \cF_2)\vert_X.
\end{equation*}

\subsection{Even $n$}

In this section we prove Theorem~\ref{theorem:an} for even~$n$.

\begin{lemma}
\label{lemma:an-even}
If $n = 2k$ then $\Db(X) = \langle \cA,\cA \otimes \cO(1,1), \dots, \cA \otimes \cO(2k-1,2k-1) \rangle$,
where~$\cA$ is defined by~\eqref{eq:ca-an}.
\end{lemma}
\begin{proof}
The $\bP^{2k-1}$-fibration $p_1$ gives rise to the semiorthogonal decomposition
\begin{equation*}
\Db(X) = \big\langle p_1^*(\Db(\bP(V))), p_1^*(\Db(\bP(V))) \otimes \cO(0,1), \dots, p_1^*(\Db(\bP(V))) \otimes \cO(0,2k-1) \big\rangle.
\end{equation*}
Choosing the exceptional collection
\begin{equation*}
\Db(\bP(V)) = \langle \cO(i-k), \cO(i-k+1), \dots, \cO(i+k) \rangle,
\end{equation*}
in the $i$-th component, we obtain a full exceptional collection in~$\Db(X)$ that takes the form
\begin{equation}
\addtolength{\jot}{-.5ex}
\label{eq:ec-an-even-1}
\begin{aligned}
\Db(X) = \Big\langle
& \cO(-k,0), \cO(1-k,0), \dots, \cO(k,0), \\
& \quad\cO(1-k,1), \cO(2-k,1), \dots, \cO(k+1,1), \\
& \quad\quad\hspace{8em} \dots \\
& \quad\quad\quad\cO(k-1,2k-1), \cO(k,2k-1), \dots, \cO(3k-1,2k-1)
\Big\rangle.
\end{aligned}
\end{equation}
The collection is shown in Picture~\ref{table:mutation-an-even}, the objects are represented by black dots.
 {\setlength\intextsep{1ex}
\begin{table}[h]
\begin{tikzpicture}[xscale = .5, yscale = .5]
\foreach \i in {0,1,2,3,4,5}
  \foreach \j in {0,1,2,3,4,5,6}
    \filldraw[black] (\j+\i,\i) circle (.2em);
\foreach \i in {0,1,2,3,4,5}
  \draw (\i+3,\i+3) -- (\i+3,\i) -- (\i+6,\i);
\foreach \j in {6,7,8} { \fill[white] (\j,6) circle (.2em); \draw[black] (\j,6) circle (.2em); }
\foreach \j in {7,8} { \fill[white] (\j,7) circle (.2em); \draw[black] (\j,7) circle (.2em); }
\fill[white] (8,8) circle (.2em); \draw[black] (8,8) circle (.2em);
\draw [rounded corners] (0.2,0.5) -- +(2,2) -- +(2,-0.7) -- +(-0.7,-0.7) -- +(0,0);
\draw [dashed, rounded corners] (6.2,6.5) -- +(2,2) -- +(2,-0.7) -- +(-0.7,-0.7) -- +(0,0);
\end{tikzpicture}
\bigskip
\renewcommand{\tablename}{Picture}
\caption{Mutation of~\eqref{eq:ec-an-even-1} to a rectangular Lefschetz collection for $k = 3$.}
\label{table:mutation-an-even}
\end{table}}

Now we perform a mutation:
we consider the subcollection formed by the first $k$ terms of the first line,
the first~\mbox{$k-1$} terms of the second line, and so on, up to the first term of the~$k$-th line of~\eqref{eq:ec-an-even-1}:
\begin{equation}
\addtolength{\jot}{-.5ex}
\label{eq:subcollection}
\begin{aligned}
\{ \cO(-k,0), \cO(1-k,0), \dots, & \cO(-1,0); \\
\cO(1-k,1), \dots, & \cO(-1,1); \\
\ddots \hphantom{,}& \qquad\vdots \\
& \cO(-1,k-1) \};
\end{aligned}
\end{equation}
(this subcollection is depicted by the triangle in the left part of Picture~\ref{table:mutation-an-even})
and mutate it to the far right of the exceptional collection.
Using Lemma~\ref{lemma:lb-vanishing} it is easy to see that the objects in~\eqref{eq:subcollection}
are right-orthogonal to the objects in the rest of~\eqref{eq:ec-an-even-1}.
Hence, their mutation to the far right is realized by the anticanonical twist (e.g. see~\cite[Lemma~2.13]{IK}).
As~$\omega_X^{-1} \cong \cO(2k,2k)$, this replaces the black dots in the triangle
by the white dots in the dashed triangle in Picture~\ref{table:mutation-an-even}.

It remains to note that the resulting full exceptional collection is precisely
the collection~$\langle \cA, \cA \otimes \cO(1,1), \dots, \cA \otimes \cO(2k-1,2k-1) \rangle$.
Indeed, the blocks formed by the twists of the subcategory~$\cA$ correspond to the L-shaped figures on the picture,
and their semiorthogonality was proved in Corollary~\ref{cor:rectangular-part-an}.
\end{proof}

\subsection{Odd $n$: rectangular part}

From now on we set $n = 2k + 1$. First, we use the trick of Lemma~\ref{lemma:an-even} to construct a full exceptional collection
that includes the rectangular part of~\eqref{eq:rectangular-an} as a subcollection.
In this case this is slightly more complicated, so we split the construction in two steps.

\begin{lemma}
\label{lemma:an-standard-ec}
The category $\Db(X)$ has the following full exceptional collection:
\begin{equation}
\addtolength{\jot}{-.5ex}
\label{eq:ec-an-1}
\begin{aligned}
\Db(X) = \Big\langle
& \cO(-k,0), \cO(1-k,0), \dots, \cO(k+1,0), \\
& \quad\cO(1-k,1), \cO(2-k,1), \dots, \cO(k+2,1), \\
& \quad\quad\hspace{8em} \dots \\
& \quad\quad\quad\cO(-1,k-1), \cO(0,k-1), \dots, \cO(2k,k-1), \\
& \quad\quad\quad\cO(-1,k), \cO(0,k), \dots, \cO(2k,k), \\
& \quad\quad\quad\quad\cO(0,k+1), \cO(1,k+1), \dots, \cO(2k+1,k+1), \\
& \quad\quad\quad\quad\quad\hspace{8em} \dots \\
& \quad\quad\quad\quad\quad\quad\cO(k-1,2k), \cO(k,2k), \dots, \cO(3k,2k)
\Big\rangle.
\end{aligned}
\end{equation}
\end{lemma}

A graphical representation for the collection in case $k = 3$ can be found in Picture~\ref{table:mutation-an}.
\begin{table}[h]
\begin{tikzpicture}[xscale = .45, yscale = .45]
\foreach \i in {0,1,2,3,4,5,6}
  \foreach \j in {0,1,2,3,4,5,6}
    \filldraw[black] (\j+\i,\i) circle (.2em);
\foreach \i in {3,4,5,6}
  \filldraw[red] (\i-1,\i) circle (.2em);
\draw[red] (2,3) circle (.5em);
\foreach \i in {0,1,2}
  \filldraw[red] (\i+7,\i) circle (.2em);
\foreach \i in {0,1,2,3,4,5,6}
  \draw (\i+3,\i+3) -- (\i+3,\i) -- (\i+6,\i);
\foreach \j in {7,8,9} { \fill[white] (\j,7) circle (.2em); \draw[black] (\j,7) circle (.2em); }
\foreach \j in {8,9} { \fill[white] (\j,8) circle (.2em); \draw[black] (\j,8) circle (.2em); }
\fill[white] (9,9) circle (.2em); \draw[black] (9,9) circle (.2em);
\draw [rounded corners] (0.2,0.5) -- +(2,2) -- +(2,-0.7) -- +(-0.7,-0.7) -- +(0,0);
\draw [dashed, rounded corners] (7.2,7.5) -- +(2,2) -- +(2,-0.7) -- +(-0.7,-0.7) -- +(0,0);
\end{tikzpicture}
\bigskip
\renewcommand{\tablename}{Picture}
\caption{Mutation of~\eqref{eq:ec-an-1} to a rectangular Lefschetz collection for $k = 3$.}
\label{table:mutation-an}
\end{table}
The objects depicted by black and red dots form the collection~\eqref{eq:ec-an-1};
the rows of~\eqref{eq:ec-an-1} correspond to rows in the picture (and the shifts of rows match up).
The mutation of Lemma~\ref{lemma:mutation-1} will take the objects corresponding to the black dots in the left triangle
to the objects corresponding to white dots in the dashed triangle at the top.
The L-shaped figures correspond to blocks of the rectangular Lefschetz collection (the first of them is~\eqref{eq:ca-an}).

\begin{proof}
The $\bP^{2k}$-fibration $p_1$ gives rise to the semiorthogonal decomposition
\begin{equation*}
\Db(X) = \big\langle p_1^*(\Db(\bP(V))), p_1^*(\Db(\bP(V))) \otimes \cO(0,1), \dots, p_1^*(\Db(\bP(V))) \otimes \cO(0,2k) \big\rangle.
\end{equation*}
This time for the first $k$ components (i.e., for $0 \le i \le k-1$) we choose the collection
\begin{equation*}
\Db(\bP(V)) = \langle \cO(i-k), \cO(i-k+1), \dots, \cO(i+k+1) \rangle,
\end{equation*}
and for the last $k + 1$ components (i.e., for $k \le i \le 2k$) we choose the collection
\begin{equation*}
\Db(\bP(V)) = \langle \cO(i-k-1), \cO(i-k), \dots, \cO(i+k) \rangle.
\end{equation*}
As the result, we obtain~\eqref{eq:ec-an-1}.
\end{proof}

It follows from Lemma~\ref{lemma:lb-vanishing} that the only nontrivial $\Ext$-spaces among the objects in Picture~\ref{table:mutation-an}
are those going in the upper-right direction
and additionally, there is non-trivial $\Ext$-space
\begin{equation}
\label{eq:ext-special}
\Ext^{2k}(\cO(2k,k-1),\cO(-1,k)) = \Bbbk
\end{equation}
from the rightmost red dot to the leftmost one.

Now we consider the same subcollection~\eqref{eq:subcollection} as in the proof of Lemma~\ref{lemma:an-even}
(this subcollection is depicted by the triangle in the left part of Picture~\ref{table:mutation-an})
and mutate it to the far right of the exceptional collection.
It follows from the description of Lemma~\ref{lemma:lb-vanishing} that the objects in~\eqref{eq:subcollection}
are right-orthogonal to the objects in the rest of~\eqref{eq:ec-an-1}.
As in Lemma~\ref{lemma:an-even}, their mutation to the far right is realized by the anticanonical twist.
As~$\omega_X^{-1} \cong \cO(2k+1,2k+1)$, we deduce the following

\begin{lemma}
\label{lemma:mutation-1}
The category $\Db(X)$ has the following full exceptional collection:
\begin{equation}
\addtolength{\jot}{-.5ex}
\label{eq:ec-an-2}
\begin{aligned}
\Db(X) = \Big\langle
& \quad\cO(0,0), \cO(1,0), \dots, \cO(k+1,0), \\
& \quad\cO(0,1), \cO(1,1), \dots, \cO(k+2,1), \\
& \quad\hspace{8em} \dots \\
& \quad\cO(0,k-1), \cO(1,k-1), \dots, \cO(2k,k-1), \\
& \cO(-1,k), \cO(0,k), \dots, \cO(2k,k), \\
& \quad\cO(0,k+1), \cO(1,k+1), \dots, \cO(2k+1,k+1), \\
& \quad\quad\hspace{8em} \dots \\
& \quad\quad\quad\cO(k-1,2k), \cO(k,2k), \dots, \cO(3k,2k), \\
& \quad\quad\quad\quad\quad\cO(k+1,2k+1), \cO(k+2,2k+1), \dots, \cO(2k,2k+1), \\
& \quad\quad\quad\quad\quad\hphantom{\cO(k+1,2k+1),}\, \cO(k+2,2k+2), \dots, \cO(2k,2k+2). \\
& \quad\quad\quad\quad\quad\hphantom{\cO(k+1,2k+1), \cO(k+2,2k+1), }\dots \\
& \quad\quad\quad\quad\quad\hphantom{\cO(k+1,2k+1), \cO(k+2,2k+1), \dots, }\, \cO(2k,3k)
\Big\rangle.
\end{aligned}
\end{equation}
\end{lemma}

The resulting collection in the case $k = 3$ is shown in Picture~\ref{table:mutation-an}:
the objects corresponding to the black dots in the lower-left triangle are replaced by the white dots in the upper-right triangle;
the rows of~\eqref{eq:ec-an-2} correspond to the rows in the picture (and the shifts of rows match).

\subsection{Odd $n$: residual category}

We note that the resulting exceptional collection~\eqref{eq:ec-an-2} is already quite close to what we need.
For instance, it already contains the rectangular part of~\eqref{eq:rectangular-an}.
Indeed, the blocks $\cA \otimes \cO(i,i)$ of~\eqref{eq:rectangular-an}
are generated by the objects corresponding to the dots in Picture~\ref{table:mutation-an} joined into L-shaped figures.

However, the objects which are not in the rectangular part
\begin{equation*}
\cO(k+1,0),\cO(k+2,1),\dots,\cO(2k,k-1),
\cO(-1,k),\cO(0,k+1),\cO(1,k+2),\dots,\cO(k-1,2k)
\end{equation*}
(they are marked with red dots on the picture),
are not yet right-orthogonal to the rectangular part (hence are not yet contained in the residual category).
The next step is to mutate them accordingly.

\begin{proposition}
If $\cR \subset \Db(X)$ is the residual category of the rectangular collection defined by~\eqref{eq:rectangular-an},
then $\cR$ is generated by the following exceptional collection
\begin{multline}
\label{eq:cr-preliminary}
\cR = \langle
\Omega^{k+1}(k+1) \boxtimes_X \cO,
\Omega^{k+2}(k+2) \boxtimes_X \cO(1), \dots,
\Omega^{2k}(2k) \boxtimes_X \cO(k-1)
,\\
\cO(-1,k),\\
\cO \boxtimes_X \Omega^{k+1}(k+1),
\cO(1) \boxtimes_X \Omega^{k+2}(k+2), \dots,
\cO(k-1) \boxtimes_X \Omega^{2k}(2k)
\rangle.
\end{multline}
The only non-trivial $\Ext$-spaces between the objects of this exceptional collection are
\begin{equation}
\label{eq:exts-cr-preliminary}
\begin{aligned}
& \Hom(\Omega^{k+i-1}(k+i-1) \boxtimes_X \cO(i-2), && \Omega^{k+i}(k+i) \boxtimes_X \cO(i-1)) 	&& = \Bbbk, && 2 \le i \le k\\
& \Hom(\Omega^{2k}(2k) \boxtimes_X \cO(k-1),	 && \cO(-1,k)) 					&& = \Bbbk,\\
& \Ext^k(\cO(-1,k),				 && \cO \boxtimes_X \Omega^{k+1}(k+1)) 		&& = \Bbbk,\\
& \Hom(\cO(i-2) \boxtimes_X \Omega^{k+i-1}(k+i-1), && \cO(i-1) \boxtimes_X \Omega^{k+i}(k+i)) 	&& = \Bbbk, && 2 \le i \le k.
\end{aligned}
\end{equation}
In particular, $\cR$ is equivalent to the derived category of the Dynkin quiver $\rA_{2k+1}$.
\end{proposition}

\begin{proof}
It follows from Lemma~\ref{lemma:lb-vanishing} that the object $\cO(-1,k)$
(corresponding to the circled red dot in Picture~\ref{table:mutation-an}) is already in the residual category,
so it is enough to mutate the other $2k$ objects.

Consider the exact sequences
\begin{multline}
\label{eq:mutation-1}
0 \to \Omega^{k+i}(k+i) \boxtimes_X \cO(i-1) \to
\Lambda^{k+i}V^\vee \otimes \cO(0,i-1) \to \\
\dots \to
V^\vee \otimes \cO(k+i-1,i-1) \to
\cO(k+i,i-1) \to 0
\end{multline}
obtained by an appropriate twist of the pullback to $X$ of the truncated Koszul complex on~$\bP(V)$,
and analogous exact sequences
\begin{multline}
\label{eq:mutation-2}
0 \to \cO(i-1) \boxtimes_X \Omega^{k+i}(k+i) \to
\Lambda^{k+i}V \otimes \cO(i-1,0) \to \\
\dots \to
V \otimes \cO(i-1,k+i-1) \to
\cO(i-1,k+i) \to 0.
\end{multline}
If we show that the objects $\Omega^{k+i}(k+i) \boxtimes_X \cO(i-1)$ and $\cO(i-1) \boxtimes_X \Omega^{k+i}(k+i)$ belong to $\cR$,
it will follow that these exact sequences express mutations of $\cO(k+i,i-1)$ and $\cO(i-1,k+i)$
through the rectangular part of~\eqref{eq:rectangular-an},
and that together with the object $\cO(-1,k)$ these objects generate the residual category~$\cR$.

So, we need to check that $\Ext^\bullet(\cO(a,b),-) = H^\bullet(X, - \otimes \cO(-a,-b))$ vanishes on these objects when $\cO(a,b)$
run through the set of objects generating the rectangular part of~\eqref{eq:rectangular-an}, i.e., for
\begin{equation}
\label{eq:lefschetz-conditions}
0 \le a \le b \le a + k \le 3k \qquad \text{and} \qquad
0 \le b \le a \le b + k \le 3k.
\end{equation}
We have exact sequences on $\bP(V) \times \bP(V^\vee)$
\begin{equation*}
0
\to \Omega^{k+i}(k+i-1) \boxtimes \cO(i-2)
\to \Omega^{k+i}(k+i) \boxtimes \cO(i-1)
\to \Omega^{k+i}(k+i) \boxtimes_X \cO(i-1)
\to 0.
\end{equation*}
To compute $H^\bullet(X,\Omega^{k+i}(k+i-a) \boxtimes_X \cO(i-1-b))$ we thus need to compute two tensor products
\begin{equation}
\begin{aligned}
\label{eq:cohomology}
H^\bullet(\bP(V), \Omega^{k+i}(k+i-a)) \otimes H^\bullet(\bP(V^\vee), \cO(i-1-b)), \\
H^\bullet(\bP(V), \Omega^{k+i}(k+i-1-a)) \otimes H^\bullet(\bP(V^\vee), \cO(i-2-b)).
\end{aligned}
\end{equation}
By Bott's formula~\cite[Proposition~14.4]{Bott} the first factors in these products are zero, except for
\begin{equation}
\label{eq:a-conditions}
\begin{aligned}
& a \le -1,
&& \text{or}
&& a = k+i,
&& \text{or}
&& a \ge 2k + 2,
\\
& a \le -2,
&& \text{or}
&& a = k+i-1,
&& \text{or}
&& a \ge 2k + 1
\end{aligned}
\end{equation}
(the conditions in the first (resp.\ second) line of~\eqref{eq:a-conditions} corresponds to
non-vanishing of the first factors in the first (resp.\ second) line in~\eqref{eq:cohomology}).
Similarly, the second factors vanish except for
\begin{equation}
\label{eq:b-conditions}
\begin{aligned}
& b \le i - 1,
&& \text{or}
&& b \ge 2k + i + 1,
\\
& b \le i - 2,
&& \text{or}
&& b \ge 2k + i
\end{aligned}
\end{equation}
(with the same convention about the role of the lines). Clearly, \eqref{eq:lefschetz-conditions} is not compatible with the conditions $a \le -1$ or $a \le -2$ in~\eqref{eq:a-conditions}.
Furthermore, \eqref{eq:lefschetz-conditions} implies that $|a - b| \le k$, which contradicts all conditions in~\eqref{eq:a-conditions} and~\eqref{eq:b-conditions}, except for those in the last columns.
However, in the latter case both $a$ and $b$ are strictly bigger than $2k$ (recall that $1 \le i \le k$),
which also contradicts~\eqref{eq:lefschetz-conditions}.

Thus, we conclude that for all $(a,b)$ satisfying~\eqref{eq:lefschetz-conditions}
either of the factors in the tensor products~\eqref{eq:cohomology} vanishes,
hence the objects $\Omega^{k+i}(k+i) \boxtimes_X \cO(i-1)$ belong to the residual category~$\cR$.
By symmetry it also follows that the objects $\cO(i-1) \boxtimes_X \Omega^{k+i}(k+i)$ belong to~$\cR$.
Thus, all the objects in the right side of~\eqref{eq:cr-preliminary} are in~$\cR$.
Moreover, as we already pointed out, it also follows that~\eqref{eq:mutation-1} and~\eqref{eq:mutation-2} are mutation sequences.
Therefore, together with the line bundle~$\cO(-1,k)$ the objects in the right side of~\eqref{eq:cr-preliminary} generate~$\cR$.

Similarly, from Lemma~\ref{lemma:lb-vanishing} and the mutation sequences, orthogonality between the rows of~\eqref{eq:cr-preliminary} follows,
so it remains to compute $\Ext$-spaces between objects in each row as well as
from objects in the first row to $\cO(-1,k)$, and from $\cO(-1,k)$ to objects in the last~row.

For this we use~\eqref{eq:mutation-1} and~\eqref{eq:mutation-2} as resolutions for the source objects.
Since almost all terms in these sequences are in the rectangular part of~\eqref{eq:rectangular-an},
it follows that it is enough to compute $\Ext$-spaces from the rightmost objects only
(i.e., for the red dots from Picture~\ref{table:mutation-an}).
Thus, we need to describe the tensor products in~\eqref{eq:cohomology} for~$(a,b)$ satisfying
\begin{equation*}
0 \le a \le k - 1,\quad
b = a + k + 1
\qquad\text{or}\qquad
0 \le b \le k - 1,\quad
a = b + k + 1.
\end{equation*}
Clearly, the assumption $0 \le a \le k - 1$ contradicts all conditions in~\eqref{eq:a-conditions}.
On the other hand, if $0 \le b \le k - 1$ and $a = b + k + 1$, then $a \le 2k$, so the only compatible conditions are
\begin{equation*}
(a,b) = (k + i, i - 1)
\qquad\text{or}\qquad
(a,b) = (k + i - 1, i - 2).
\end{equation*}
Furthermore, in the first of these cases the first tensor product in~\eqref{eq:cohomology} is 1-dimensional
(and lives in the cohomological degree~$k+i$),
and in the second of these cases the second tensor product is 1-dimensional
(and lives in the same cohomological degree).
This computation proves that the only $\Ext$-spaces between the objects in the first and the last rows of~\eqref{eq:cr-preliminary}
are those listed in the first and last lines of~\eqref{eq:exts-cr-preliminary}.

Next, we compute $\Ext$-spaces from objects in the first row of~\eqref{eq:cr-preliminary} to~$\cO(-1,k)$.
As before we use~\eqref{eq:mutation-1} as resolutions.
The only object that appears in~\eqref{eq:mutation-1}
which is not semiorthogonal to $\cO(-1,k)$ is $\cO(2k,k-1)$ (see Lemma~\ref{lemma:lb-vanishing} and~\eqref{eq:ext-special}),
from which we have a 1-dimensional $\Ext$-space in degree~$2k$.
Therefore, the only $\Ext$-space from objects of the first row to~$\cO(-1,k)$
is given by the second line of~\eqref{eq:exts-cr-preliminary}.

Finally, we compute $\Ext$-spaces from~$\cO(-1,k)$ to objects in the last row of~\eqref{eq:cr-preliminary}.
First, we consider the Koszul exact sequence
\begin{equation*}
0 \to \cO(-1,k) \to
\Lambda^{2k+1}V^\vee \otimes \cO(0,k) \to \\
\dots \to
V^\vee \otimes \cO(2k,k) \to
\cO(2k+1,k) \to 0.
\end{equation*}
All its terms (except for the leftmost and rightmost) are in the rectangular part of~\eqref{eq:rectangular-an}, hence
\begin{equation*}
\Ext^p(\cO(-1,k), \cO(i-1) \boxtimes_X \Omega^{k+i}(k+i)) \cong
\Ext^{p + 2k + 1}(\cO(2k+1,k), \cO(i-1) \boxtimes_X \Omega^{k+i}(k+i)).
\end{equation*}
Now we use~\eqref{eq:mutation-2} as a resolution for $\cO(i-1) \boxtimes_X \Omega^{k+i}(k+i)$.
Clearly, the only nontrivial $\Ext$-space from $\cO(2k+1,k)$ to its terms is
\begin{equation*}
\Ext^{2k}(\cO(2k+1,k),\cO(0,k+1)) \cong H^{2k}(X,\cO(-2k-1,1)) \cong \Bbbk
\end{equation*}
(which holds for $i = k + 1$). Therefore, the only $\Ext$-space from $\cO(-1,k)$ to the last row of~\eqref{eq:rectangular-an}
is given by the third line of~\eqref{eq:exts-cr-preliminary}.

It is clear from the above description of $\Ext$-spaces
that the category $\cR$ is equivalent to the derived category of the Dynkin quiver~$\rA_{2k+1}$
(with the objects in~\eqref{eq:cr-preliminary} corresponding to simple representations of the quiver, up to shift).
\end{proof}

So far, the description of $\cR$ we obtained does not look symmetric with respect to outer automorphisms of $X$,
and also looks different from the description in Theorem~\ref{theorem:an}.
In the next statement we show that the two descriptions agree.

\begin{lemma}
\label{lemma:residual-an}
The category~$\cR$ defined by~\eqref{eq:cr-preliminary} is generated by the exceptional collection~\eqref{eq:residual-final}.
\end{lemma}

\begin{proof}
The definition~\eqref{def:ce} of the object $\cE$ can be rewritten as the exact sequence
\begin{equation}
\label{eq:ce-1}
0 \to \cE \to \rT(-1) \boxtimes_X \cO \to \cO(0,1) \to 0.
\end{equation}
It follows that $\rk(\cE) = 2k$ and $\det(\cE) \cong \cO(1,-1)$, hence
\begin{equation}
\label{eq:wedge-ce}
\Lambda^i\cE^\vee \cong \Lambda^{2k-i}\cE \otimes \cO(-1,1).
\end{equation}
Dualizing~\eqref{eq:ce-1}, we obtain
\begin{equation*}
0 \to \cO(0,-1) \to \Omega(1) \boxtimes_X \cO \to \cE^\vee \to 0.
\end{equation*}
Taking its $(k+i)$-th exterior power and twisting by $\cO(0,i-1)$, we obtain
\begin{equation*}
0 \to \Lambda^{k+i-1}\cE^\vee(0,i-2) \to \Omega^{k+i}(k+i) \boxtimes_X \cO(i-1) \to \Lambda^{k+i}\cE^\vee(0,i-1) \to 0.
\end{equation*}
Using~\eqref{eq:wedge-ce}, we can rewrite this as
\begin{equation}
\label{eq:ce-om-ce}
0 \to \Lambda^{k-i+1}\cE(-1,i-1) \to \Omega^{k+i}(k+i) \boxtimes_X \cO(i-1) \to \Lambda^{k-i}\cE(-1,i) \to 0.
\end{equation}
Now we interpret these sequences as mutations of~\eqref{eq:cr-preliminary}:
\begin{itemize}
\item
For $i = k$, the second arrow in~\eqref{eq:ce-om-ce} is the unique (by~\eqref{eq:exts-cr-preliminary}) morphism
\begin{equation*}
\Omega^{2k}(2k) \boxtimes_X \cO(k-1) \to \cO(-1,k),
\end{equation*}
therefore its kernel (up to shift) is the right mutation of $\Omega^{2k}(2k) \boxtimes_X \cO(k-1)$ through~$\cO(-1,k)$,
and so by~\eqref{eq:ce-om-ce} the result of the mutation is $\cE(-1,k-1)$.
\item
For $i = k-1$, the second arrow in~\eqref{eq:ce-om-ce} is the unique morphism
\begin{equation*}
\Omega^{2k-1}(2k-1) \boxtimes_X \cO(k-2) \to \cE(-1,k-1),
\end{equation*}
therefore its kernel (up to shift) is the right mutation of $\Omega^{2k-1}(2k-1) \boxtimes_X \cO(k-2)$ through $\cE(-1,k-1)$,
and so the result of the mutation is $\Lambda^2\cE(-1,k-2)$.
Note here that the intermediate mutation of $\Omega^{2k-1}(2k-1) \boxtimes_X \cO(k-2)$ through $\cO(-1,k)$ is trivial,
because these objects are completely orthogonal by~\eqref{eq:exts-cr-preliminary}.
\end{itemize}
Continuing in the same manner, we conclude that the mutation of $\Omega^{2k-j}(2k-j) \boxtimes_X \cO(k-j-1)$
through $\langle \Omega^{2k-j+1}(2k-j+1) \boxtimes_X \cO(k-j), \dots, \Omega^{2k}(2k) \boxtimes_X \cO(k-1), \cO(-1,k) \rangle$
is the same as its mutation through $\langle \cO(-1,k), \cE(-1,k-1), \dots, \Lambda^j\cE(-1,k-j) \rangle$,
and is the same as its mutation through the last object $\Lambda^j\cE(-1,k-j)$.
Therefore, it is realized by the complex~\eqref{eq:ce-om-ce} with~$i = k - j$,
hence the result of the mutation is $\Lambda^{j+1}\cE(-1,k-j-1)$.

Thus, we see that the category $\cR$ is generated by the following exceptional collection
\begin{multline}
\label{eq:cr-intermediate}
\cR = \langle \cO(-1,k), \cE(-1,k-1), \dots, \Lambda^{k-1}\cE(-1,1), \Lambda^k\cE(-1,0),
\\
\cO \boxtimes_X \Omega^{k+1}(k+1),
\cO(1) \boxtimes_X \Omega^{k+2}(k+2), \dots,
\cO(k-1) \boxtimes_X \Omega^{2k}(2k)
\rangle,
\end{multline}
and moreover, the only $\Ext$-spaces between the objects in the first row and the second row of~\eqref{eq:cr-intermediate}, are
\begin{equation}
\label{eq:ext-cr-intermediate}
\Ext^\bullet(\Lambda^i\cE(-1,k-i),\cO(j-1) \boxtimes_X \Omega^{k+j}(k+j)) =
\begin{cases}
\Bbbk[i-k], & \text{if $j = 1$}\\
0, & \text{if $2 \le j \le k$.}
\end{cases}
\end{equation}
Note also that by~\eqref{eq:wedge-ce} the last term in the first line of~\eqref{eq:cr-intermediate}
is isomorphic to $\Lambda^k\cE^\vee(0,-1)$.

For the second half of mutations, we note that applying an outer automorphism of $X$ to~\eqref{eq:ce-om-ce} we obtain exact sequences
\begin{equation}
\label{eq:ce-om-ce-2}
0 \to \Lambda^{k-i+1}\cE^\vee(i-1,-1) \to \cO(i-1) \boxtimes_X \Omega^{k+i}(k+i) \to \Lambda^{k-i}\cE^\vee(i,-1) \to 0.
\end{equation}
Now we interpret these sequences as mutations. 
Note that~$\Lambda^{k}\cE(-1,0) \cong \Lambda^{k}\cE^\vee(0,-1)$ by~\eqref{eq:wedge-ce}).
\begin{itemize}
\item
For $i = 1$, the first arrow in~\eqref{eq:ce-om-ce-2} is the unique (by~\eqref{eq:ext-cr-intermediate}) morphism
\begin{equation*}
\Lambda^{k}\cE^\vee(0,-1) \to \cO \boxtimes_X \Omega^{k+1}(k+1),
\end{equation*}
therefore its cokernel (up to shift) is the left mutation of the object $\cO \boxtimes_X \Omega^{k+1}(k+1)$ through~$\Lambda^{k}\cE^\vee(0,-1)$,
and so the result of the mutation is $\Lambda^{k-1}\cE^\vee(1,-1)$.
\item
For $i = 2$, the first arrow in~\eqref{eq:ce-om-ce-2} is the unique morphism
\begin{equation*}
\Lambda^{k-1}\cE^\vee(1,-1) \to \cO(1) \boxtimes_X \Omega^{k+2}(k+2),
\end{equation*}
therefore its cokernel (up to shift) is the left mutation of $\cO(1) \boxtimes_X \Omega^{k+2}(k+2)$ through~$\Lambda^{k-1}\cE^\vee(1,-1)$,
and so the result of the mutation is $\Lambda^{k-2}\cE^\vee(2,-1)$.
Note here that the intermediate mutation of $\cO(1) \boxtimes_X \Omega^{k+2}(k+2)$ through $\Lambda^{k}\cE^\vee(0,-1)$ is trivial,
because these objects are completely orthogonal by~\eqref{eq:ext-cr-intermediate}.
\end{itemize}
Continuing in the same manner, we conclude that the mutation of $\cO(i-1) \boxtimes_X \Omega^{k+i}(k+i)$
through $\langle \Lambda^{k}\cE^\vee(0,-1), \cO \boxtimes_X \Omega^{k+1}(k+1), \dots, \cO(i-2) \boxtimes_X \Omega^{k+i-1}(k+i-1) \rangle$
is the same as its mutation through $\langle \Lambda^{k-i+1}\cE^\vee(i-1,-1), \dots, \Lambda^{k-1}\cE^\vee(1,-1), \Lambda^{k}\cE^\vee(0,-1) \rangle$,
and the same as its mutation through the first object $\Lambda^{k-i+1}\cE^\vee(i-1,-1)$.
Therefore, it is realized by the complex~\eqref{eq:ce-om-ce-2},
hence the result of the mutation is $\Lambda^{k-i+1}\cE^\vee(i-1,-1)$.

Thus, we see that the category $\cR$ is generated by~\eqref{eq:residual-final}.
Moreover, we see that the only $\Ext$-spaces between the objects in~\eqref{eq:residual-final} are
\begin{align*}
&\Ext^{j-i}(\Lambda^i\cE\hphantom{{}^\vee}(0,k-i), \Lambda^j\cE\hphantom{{}^\vee}(0,k-j)) = \Bbbk, && 0 \le i \le j \le k,
\\
&\Ext^{j-i}(\Lambda^i\cE^\vee(0,k-i), \Lambda^j\cE^\vee(0,k-j)) = \Bbbk, && 0 \le i \le j \le k.
\end{align*}
In other words, the configuration of these exceptional objects is the following:
\begin{equation*}
\xymatrix@C=.8em@R=1ex{
\cO(-1,k)[-k] \ar[r] &
\cE(-1,k-1)[1-k] \ar[r] &
\dots \ar[r] &
\Lambda^{k-1}\cE(-1,1)[-1] \ar[dr]
\\
&&&& \Lambda^k\cE(-1,0) \cong \Lambda^k\cE^\vee(0,-1)
\\
\cO(k,-1)[-k] \ar[r] &
\cE^\vee(k-1,-1)[1-k] \ar[r] &
\dots \ar[r] &
\Lambda^{k-1}\cE^\vee(1,-1)[-1] \ar[ur]
}
\end{equation*}
So, this exceptional collection corresponds to shifts of projective modules in a quiver of Dynkin type $\rA_{2k+1}$.
\end{proof}

\begin{remark}
\label{remark:an-weights}
The flag variety $X$ is a homogeneous space for the action of the reductive group~$\SL(n+1) = \SL(V)$.
Using the group-theoretic notation introduced in~\S\ref{section:preliminaries},
and denoting by~$\omega_i$ the highest weight of the fundamental representation~$\Lambda^iV^\vee \cong \Lambda^{n+1-i}V$,
one can rewrite the exceptional collection of Theorem~\ref{theorem:an} as follows.
The first block of the rectangular part~$\cA$ of $\Db(X)$ can be written as
\begin{equation*}
  \cA = \langle \cO, \cU^{\omega_1}, \cU^{\omega_n}, \cU^{2\omega_1}, \cU^{2\omega_n}, \dots,  \cU^{k\omega_1}, \cU^{k\omega_n} \rangle
\end{equation*}
Furthermore, if $n = 2k + 1$ the residual category $\cR$ can be written as
\begin{multline*}
\cR = \langle
\cU^{-\omega_1 + k\omega_{n} }, \cU^{k\omega_1 - \omega_{n}};
\cU^{-\omega_1 + \omega_{n-1} + (k-2)\omega_{n}}, \cU^{(k-2)\omega_1 + \omega_2 - \omega_{n}}; \dots \\
\dots;
\cU^{-\omega_1 + \omega_{n-i} + (k-i-1)\omega_{n}}, \cU^{(k-i-1)\omega_1 + \omega_{i+1} - \omega_{n}};
\dots \\
\cU^{- \omega_1 + \omega_{k+2} }, \cU^{\omega_{k} - \omega_{n}};
\cU^{- \omega_1 + \omega_{k+1}  - \omega_{n}}
\rangle.
\end{multline*}
\end{remark}

\section{Type $\rD_n$}
\label{section:dn}

In this section we prove part~\eqref{theorem:intro:dn} of Theorem~\ref{theorem:intro},
restated in a more precise form in Theorem~\ref{theorem:dn} below.
We work over an algebraically closed field~$\Bbbk$ of characteristic zero.
Let us point out again that, unlike in type~$\rA_n$, a full exceptional collection in~$\Db(\OG(2,2n))$ was not known before.
Throughout this section we assume~$n \ge 4$.

\subsection{Equivariant bundles on homogeneous varieties}
\label{section:preliminaries}

We start with a brief reminder about equivariant vector bundles on homogeneous varieties.
For more details see~\cite{KuPo}.

Let $\rG$ be a connected simply connected semisimple algebraic group. We fix a Borel subgroup $\rB \subset \rG$.
Let $\rP \subset \rG$ be a parabolic subgroup such that $\rB \subset \rP$. Recall that there is a monoidal equivalence of categories
between the category of~$\rG$-equivariant vector bundles on the homogeneous variety~$\rG/\rP$
and the category~$\Rep(\rP)$ of representations of the parabolic subgroup $\rP$.

Let $\rP \twoheadrightarrow \rL$ be the Levi quotient.
We identify the category~$\Rep(\rL)$ of representations of~$\rL$ with the subcategory of~$\Rep(\rP)$ of representations
with the trivial action of the unipotent radical.
This equivalence is compatible with the monoidal structure of the categories.

The image of the Borel subgroup $\rB \subset \rG$ is a Borel subgroup~$\rB_\rL$ in~$\rL$.
We denote by~\mbox{$\bfP_\rL = \bfP_\rG$} the weight lattices of $\rL$ and $\rG$ with their natural identification, and by
\begin{equation*}
\bfP_\rG^+ \subset \bfP_\rL^+
\end{equation*}
the cones of dominant weights with respect to~$\rB$ and~$\rB_\rL$, respectively.
For each $\lambda \in \bfP_\rL^+$
we denote by $\rV_\rL^\lambda$
the corresponding irreducible representation of $\rL$
and by $\cU^\lambda$ the equivariant vector bundle on $\rG/\rP$ corresponding to it via the above equivalences.
Similarly, for $\lambda \in \bfP_\rG^+$ we denote by $\rV_\rG^\lambda$ the corresponding irreducible representation of $\rG$.

We denote by~$\rW$ the Weyl group of $\rG$, by $\rW^\rL \subset \rW$ the Weyl group of $\rL$,
and by~$w_0 \in \rW$ and~$w_0^\rL \in \rW^\rL$ the longest elements.

\subsection{Statement of the theorem}
\label{subsection:theorem-dn}

Consider the group~\mbox{$\rG = \Spin(2n)$}, the simply connected (double) covering
\begin{equation}
\label{eq:double-covering}
\Spin(2n) \to \SO(2n)
\end{equation}
of the special orthogonal group, and its maximal parabolic subgroup $\rP \subset \rG$ corresponding to the vertex~2 (marked with black) of the Dynkin diagram $\rD_n$
\begin{equation}
\label{eq:dn}
\vcenter{\hbox{\begin{picture}(200,30)
\multiput(0,10)(30,0){3}{\circle{4}}
\put(30,10){\circle*{4}}
\put(90,10){\hbox to 0mm{\hss\dots\hss}}
\put(120,10){\circle{4}}
\put(150,10){\circle{4}}
\put(180,20){\circle{4}}
\put(180,0){\circle{4}}
\multiput(2,10)(30,0){2}{\line(1,0){26}}
\put(62,10){\line(1,0){10}}
\put(108,10){\line(1,0){10}}
\put(122,10){\line(1,0){26}}
\put(152,11){\line(3,1){26}}
\put(152,9){\line(3,-1){26}}
\put(-2,15){$\scriptstyle{}1$}
\put(28,15){$\scriptstyle{}2$}
\put(58,15){$\scriptstyle{}3$}
\put(112,15){$\scriptstyle{}n-3$}
\put(142,15){$\scriptstyle{}n-2$}
\put(184,20){$\scriptstyle{}n-1$}
\put(184,0){$\scriptstyle{}n$}
\end{picture}}}
\end{equation}
In this case the corresponding homogeneous variety is
\begin{equation*}
\rG/\rP \cong \OG(2,2n),
\end{equation*}
the Grassmannian of 2-dimensional isotropic subspaces in a vector space of dimension~$2n$
endowed with a non-degenerate symmetric bilinear form, i.e., the (co)adjoint variety of type~$\rD_n$.
Note that $\dim(\OG(2,2n)) = 4n - 7$.

The automorphism group of $\OG(2,2n)$ is the semidirect product
\begin{equation*}
\Aut(\OG(2,2n)) \cong
\begin{cases}
\PSO(2n) \rtimes \fS_2, & \text{if $n \ne 4$}\\
\PSO(2n) \rtimes \fS_3, & \text{if $n = 4$}\\
\end{cases}
\end{equation*}
where the factors $\fS_2$ and $\fS_3$ act by outer automorphisms (corresponding to the symmetry of the Dynkin diagram $\rD_n$).

The weight lattice of $\SO(2n)$ is the lattice $\bZ^n$ with the standard basis~$\{\epsilon_i\}_{1 \le i \le n}$,
and the weight\ lattice of~$\Spin(2n)$ is its overlattice in~$\bQ^n$ generated by the fundamental weights
\begin{equation}
\label{eq:dn-weights}
\begin{aligned}
\omega_i &= \epsilon_1 + \dots + \epsilon_i, && 1 \le i \le n-2, \\
\omega_{n-1} &= \tfrac12(\epsilon_1 + \dots + \epsilon_{n-1} - \epsilon_n), \\
\omega_{n} &= \tfrac12(\epsilon_1 + \dots + \epsilon_{n-1} + \epsilon_n).
\end{aligned}
\end{equation}

The Levi group in $\SO(2n)$ corresponding to the second vertex of the Dynkin diagram is isomorphic
to $\GL(2) \times \SO(2(n-2))$,
and the corresponding Levi group in $\Spin(2n)$ is the double covering
\begin{equation}
\label{eq:levi-covering}
\rL \to \GL(2) \times \SO(2(n-2))
\end{equation}
induced by~\eqref{eq:double-covering}.

The fundamental weight $\omega_2$ (associated with the parabolic $\rP$) corresponds to the ample generator of the Picard group $\Pic(\rG/\rP)$,
so we write
\begin{equation}
\label{eq:dn-o1}
\cO(1) := \cU^{\omega_2}.
\end{equation}
Note that the canonical line bundle can be written as
\begin{equation}
\label{eq:omega-dn}
\omega_{\OG(2,2n)} \cong \cO_{\OG(2,2n)}(3 - 2n).
\end{equation}
Furthermore, we use the following notation
\begin{equation}
\label{eq:bundles-dn}
\begin{aligned}
\cU &:= (\cU^{\omega_1})^\vee,\\
\cS_- &:= \cU^{\omega_{n-1}},\\
\cS_+ &:= \cU^{\omega_{n}}.
\end{aligned}
\end{equation}
Thus, $\cU$ is the tautological rank-2 bundle of $\OG(2,2n)$, while $\cS_\pm$ are the \emph{spinor bundles} (see~\cite[\S6]{Ku08a}).

Consider the following full triangulated subcategories of $\Db(\OG(2,2n))$
\begin{equation}
\label{eq:a-b-dn}
\begin{aligned}
\cA &:= \langle \cO , \cU^\vee, S^2 \cU^\vee, \dots , S^{n-3} \cU^\vee, \cS_- , \cS_+ \rangle, \\
\cB &:= \langle \cO , \cU^\vee, S^2 \cU^\vee, \dots , S^{n-3} \cU^\vee, S^{n-2} \cU^\vee, \cS_- , \cS_+ \rangle.
\end{aligned}
\end{equation}
Note that $\cA \subset \cB$ and $\cA$ is $\Aut(\OG(2,2n))$-invariant.
Indeed, every object in~$\cA$ is $\PSO(2n)$-invariant by~\cite[Lemma~2.2]{Po}.
Moreover, if $n \ne 4$ the outer automorphisms $\fS_2$-action on~$\OG(2,2n)$ swaps the spinor bundles $\cS_-$ and~$\cS_+$,
while in case $n = 4$ the outer automorphisms $\fS_3$-action permutes $\cS_-$, $\cS_+$, and~$ \cU^\vee$
(and in this case $\cA = \langle \cO , \cU^\vee, \cS_- , \cS_+ \rangle$).

\begin{theorem}
\label{theorem:dn}
There exists a full exceptional collection
\begin{equation}
\label{eq-collection-OG-even-I}
\begin{split}
\Db(\OG(2,2n)) = \langle & \cU^{2\omega_{n-1}}(-1), \cU^{2\omega_{n}}(-1) , \cA , \\
& \cB(1) , \dots , \cB(n-2) , \cA(n-1) , \dots ,\cA(2n-4) \rangle.
\end{split}
\end{equation}
In particular, the subcategory $\cA$ extends to an $\Aut(\OG(2,2n))$-invariant rectangular Lefschetz collection
\begin{equation}
\label{eq:rectangular-dn}
\Db(\OG(2,2n)) = \langle \cR, \cA, \cA(1), \dots, \cA(2n - 4) \rangle.
\end{equation}
Moreover, the residual category $\cR$ is generated by an~$\Aut(\OG(2,2n))$-invariant exceptional collection and
is equivalent to the derived category of representations of the Dynkin quiver~$\rD_n$.
\end{theorem}

\begin{remark}
It is easy to see that the exceptional collection~\eqref{eq-collection-OG-even-I} is mutation equivalent to the full Lefschetz collection with the starting block
\begin{equation*}
\begin{split}
E_i & = S^{i-1} \cU^\vee,\quad 1 \le i \le n-2, \\
E_{n-1} & = \cS_-, \\
E_{{n}} & = \cS_+, \\
E_{n+1} &= \bR_{\langle \cS_-, \cS_+ \rangle}(S^{n-2}\cU^\vee), \\
E_{n+2} &= \bR_{\langle \cA(-1),\cB \rangle} (\cU^{2\omega_{n-1}}(-2)), \\
E_{n+3} &= \bR_{\langle \cA(-1),\cB \rangle} (\cU^{2\omega_{n}}(-2)),
\end{split}
\end{equation*}
with the next $n-3$ blocks generated by~$E_1\dots,E_{n+1}$, and with the last $n-1$ blocks generated by~$E_1,\dots,E_n$.
Indeed, to prove this we first twist~\eqref{eq-collection-OG-even-I} by~$\cO(-1)$,
then mutate the objects~$\cU^{2\omega_{n-1}}(-2), \cU^{2\omega_{n}}(-2)$ to the right of~$\cA(-1)$ and~$\cB$,
and finally mutate~$\cA(-1)$ to the far right.
\end{remark}

\begin{remark}
If $n = 3$ we have isomorphisms $\Spin(6) \cong \SL(4)$ and $\OG(2,6) \cong \Fl(1,3;4)$.
Furthermore, under these isomorphisms $\cS_- \cong \cO(1,0)$ and $\cS_+ \cong \cO(0,1)$,
so the definitions of the category~$\cA$ in~\eqref{eq:ca-an} and~\eqref{eq:a-b-dn} coincide,
hence~\eqref{eq:rectangular-an} coincides with~\eqref{eq:rectangular-dn}.
\end{remark}

We prove Theorem~\ref{theorem:dn} in Sections~\ref{subsection:exceptional-dn}, \ref{subsection:fullness-dn} and~\ref{subsection:residual-dn}
after some preparation.

\subsection{Computational tools}
\label{subsection:computational-dn}

To check that a collection of vector bundles is exceptional, we need some tools to dualize the bundles,
take their tensor products, and compute the cohomology.

We start by reminding the general machinery.
Recall that the Levi group $\rL$ is reductive, hence $\Rep(\rL)$ is a semisimple category;
in particular $\rV_\rL^\lambda \otimes \rV_\rL^\mu$ is a direct sum of irreducible representations of~$\rL$.
Since the functor $\rV_\rL^\lambda \mapsto \cU^\lambda$
from the category $\Rep(\rL)$ to the category of equivariant vector bundles on $\rG/\rP$ is monoidal, we have the following

\begin{lemma}[{\cite[(8)]{KuPo}}]
\label{lemma:tp-general}
If $\rV_\rL^\lambda \otimes \rV_\rL^\mu = \bigoplus \rV_\rL^\nu$, then $\cU^\lambda \otimes \cU^\mu = \bigoplus \cU^\nu$.
\end{lemma}

Thus, knowing tensor products of $\rL$-representations, we can control tensor products of the corresponding equivariant bundles.
Similarly, we can control the dualization operation.
Recall that~$w_0^\rL$ denotes the longest element in $\rW^\rL$, the Weyl group of~$\rL$.

\begin{lemma}[{\cite[(8)]{KuPo}}]
\label{lemma:duals-general}
We have $(\cU^\lambda)^\vee \cong \cU^{-w_0^\rL \lambda}$.
\end{lemma}

A combination of the last two lemmas shows that $(\cU^\lambda)^\vee \otimes \cU^\mu$ is a direct sum of some $\cU^\nu$.
The summands that show up in this decomposition can be characterised by the following

\begin{lemma}[{\cite[Lemma~2.8 and Lemma~2.9]{KuPo}}]
\label{lemma:prv}
We have
\begin{equation*}
(\cU^\lambda)^\vee \otimes \cU^\mu =
\bigoplus_{\nu \in \bfP_\rL^+ \cap {} \Conv(\mu - w\lambda)_{w \in \rW_\rL}} (\cU^\nu)^{\oplus m(\lambda,\mu,\nu)},
\end{equation*}
where $\Conv(-)$ stands for the convex hull and $m(\lambda,\mu,\nu) \in \bZ_{\ge 0}$
is the multiplicity of $\rV^\nu_\rL$ in~$(\rV_\rL^\lambda)^\vee \otimes \rV_\rL^\mu$.
Moreover, if $\nu = 0$ then $m(\lambda,\mu,\nu) \ne 0$ only when $\lambda = \mu$ and~\mbox{$m(\lambda,\lambda,0) = 1$}.
\end{lemma}

The most efficient way to compute cohomology is provided by the Borel--Bott--Weil theorem.
Let $\ell \colon \rW \to \bZ$ be the length function
and let $\rho \in \bfP_\rG^+$ be the sum of fundamental weights of~$\rG$.

\begin{theorem}[Borel--Bott--Weil]
Let $\lambda \in \bfP_\rL^+$.
If the weight $\lambda + \rho$ lies on a wall of a Weyl chamber for the $\rW$-action, then
\begin{equation*}
H^\bullet(\rG/\rP,\cU^\lambda) = 0.
\end{equation*}
Otherwise, if $w \in \rW$ is the unique element such that the weight $w(\lambda + \rho)$ is dominant, then
\begin{equation*}
H^\bullet(\rG/\rP,\cU^\lambda) = \rV_\rG^{w(\lambda + \rho) - \rho}[-\ell(w)].
\end{equation*}
\end{theorem}

The Weyl group of $\Spin(2n)$ is the semidirect product
\begin{equation*}
\rW = \fS_n \ltimes (\bZ/2)^{n-1},
\end{equation*}
where $\fS_n$ acts on $\bfP_{\Spin(2n)} {} \subset \bQ^n$ by permutations
and $(\bZ/2)^{n-1}$ by changes of signs of even number of coordinates.
The walls of the Weyl chambers are given by the hyperplanes
\begin{equation}
\label{eq:walls-dn}
\lambda_i = \pm \lambda_j,\qquad
1 \le i \ne j \le n.
\end{equation}
The Weyl group of $\rL$ is the subgroup
\begin{equation*}
\rW^\rL = \fS_2 \times (\fS_{n-2} \ltimes (\bZ/2)^{n-3}),
\end{equation*}
where $\fS_2$ acts by transposition of the first two coordinates, $\fS_{n-2}$ by permutations of the last $n-2$,
and $(\bZ/2)^{n-3}$ by changes of signs of even number of the last $n-2$ coordinates.
The longest element $w_0^\rL \in \rW^\rL$ acts by
\begin{equation}
\label{eq:w0l-type-d}
w_0^\rL(\lambda_1,\lambda_2,\lambda_3,\dots,\lambda_{n-1},\lambda_n) = (\lambda_2,\lambda_1,-\lambda_3,\dots,-\lambda_{n-1},-(-1)^n\lambda_n).
\end{equation}
Finally, the sum $\rho$ of the fundamental weights in the standard basis takes the form
\begin{equation}
\label{eq:rho-dn}
\rho = (n - 1, n - 2, \dots, 1, 0).
\end{equation}

Applying the general machinery in our situation we obtain the following corollaries.
First, combining Lemma~\ref{lemma:duals-general} with~\eqref{eq:dn-weights}, \eqref{eq:dn-o1}, and~\eqref{eq:w0l-type-d}, we obtain

\begin{corollary}
\label{corollary:duals}
For $a \geq 0$ we have
\begin{equation*}
\big( \cU^{a \omega_1} \big)^\vee = \cU^{a \omega_1}(-a).
\end{equation*}
For $a \geq 0$ and even $n$ we have
\begin{equation*}
\big(\cU^{a\omega_{n-1}} \big)^\vee = \cU^{a\omega_{n-1}}(-a) \qquad \text{and} \qquad \big( \cU^{a\omega_{n}} \big)^\vee = \cU^{a\omega_{n}}(-a).
\end{equation*}
For $a \geq 0$ and odd $n$ we have
\begin{equation*}
\big( \cU^{a\omega_{n-1}} \big)^\vee = \cU^{a\omega_{n}}(-a) \qquad \text{and} \qquad \big( \cU^{a\omega_{n}} \big)^\vee = \cU^{a\omega_{n-1}}(-a).
\end{equation*}
\end{corollary}

\begin{proof}
By~\eqref{eq:dn-weights} we have $a\omega_1 = (a,0,0,\dots,0)$.
By Lemma~\ref{lemma:duals-general} and~\eqref{eq:w0l-type-d} the weight corresponding to~$\big( \cU^{a \omega_1} \big)^\vee$
is~$(0,-a,0,\dots,0)$, by~\eqref{eq:dn-weights} this is equal to~$a\omega_1 - a\omega_2$, hence the claim.
The other results are proved analogously.
\end{proof}

Similarly, using Lemma~\ref{lemma:tp-general} we deduce the following tensor product decompositions.

\begin{lemma}
\label{lemma-tensor-products}
We have isomorphisms
\begingroup\allowdisplaybreaks
\begin{align}
\label{tp:a1*b1}
\cU^{k \omega_1} \otimes \cU^{l \omega_1}
&\cong \bigoplus_{i = 0}^{\min\{k,l\}} \cU^{(k+l-2i) \omega_1} (i)
&& \text{for $k,l \ge 0$,}
\\
\label{tp:a1-la}
\cU^{k\omega_1} \otimes \cU^\lambda
&\cong \cU^{k\omega_1 + \lambda}
&& \text{if $\lambda = \sum_{i=2}^n \lambda_i \omega_i$,}
\\
\label{eq:wedge-n-2}
\Lambda^{n-2}(\cU^{\omega_3 - \omega_2})
&\cong
\cU^{2\omega_{n-1}}(-1) \oplus \cU^{2\omega_{n}}(-1).
\end{align}
\endgroup
\end{lemma}

\begin{proof}
First, we note that the bundles $\cU^{k\omega_1}$ and~$\cU^{l\omega_1}$
correspond to representations of $\rL$ pulled back from representations $\rV_{\GL_2}^{k,0}$ and $\rV_{\GL_2}^{l,0}$
via the map~\eqref{eq:levi-covering}, so~\eqref{tp:a1*b1} follows from the Clebsch--Gordan rule
\begin{equation*}
\rV_{\GL_2}^{k_1,k_2} \otimes \rV_{\GL_2}^{l_1,l_2} =
\bigoplus_{j = 0}^{\min\{k_1 - k_2,l_1 - l_2\}} \rV_{\GL_2}^{k_1 + l_1 - j, k_2 + l_2 + j}.
\end{equation*}

For~\eqref{tp:a1-la} and~\eqref{eq:wedge-n-2} we need a slightly more detailed understanding of the Levi group~$\rL$.
Let $\tGL(2) \to \GL(2)$ be the unique connected double covering.
Then there is a double covering
\begin{equation}
\label{eq:levi-covering-2}
\tGL(2) \times \Spin(2(n-2)) \to \rL.
\end{equation}
To understand tensor product of representations of~$\rL$ it is enough to understand tensor product
of their pullbacks via the map~\eqref{eq:levi-covering-2}.

We will denote by $\omega'_1$, $\omega'_2$ the fundamental weights of $\tGL(2)$
and by $\omega'_i$, $3 \le i \le n$, the fundamental weights of $\Spin(2(n-2))$.
The double covering~\eqref{eq:levi-covering-2} induces an embedding of weight lattices
\begin{equation*}
 \bfP_\rL \hookrightarrow \bfP_{\tGL(2)} \oplus \bfP_{\Spin(2(n-2))}
\end{equation*}
and identifies the corresponding~$\bQ$-vector spaces.
Thereofre, the weights~$\omega'_i$
can be thought of as rational weights for $\rL$, i.e., as rational weights of $\Spin(2n)$.
It is easy to see that they can be expressed in the basis of fundamental weights $\omega_i$ as follows:
\begin{equation}
\label{eq:omega-prime}
\omega'_1 = \omega_1,\quad
\omega'_2 = \tfrac12\omega_2,\qquad
\omega'_i =
\begin{cases}
\omega_i - \omega_2, & \text{if $3 \le i \le n-2$,}\\
\omega_i - \tfrac12\omega_2, & \text{if $i \in \{n-1,n\}$.}
\end{cases}
\end{equation}
By~\eqref{eq:omega-prime} the representation $\rV_\rL^{k\omega_1}$ pulls back to~$\rV_{\tGL_2}^{k,0}$,
while $\rV_\rL^\lambda$ with $\lambda = \sum_{i=2}^n \lambda_i \omega_i$
pulls back to~$\rV_{\tGL(2)}^{c(\lambda),c(\lambda)} \otimes \rV_{\Spin(2(n-2))}^{\lambda'}$,
where $\lambda' = \sum_{i=3}^n \lambda_i \omega'_i$ and~$c(\lambda) = 2\sum_{i=2}^{n-2}\lambda_i + \lambda_{n-1} + \lambda_n$.
Clearly,
\begin{equation*}
\rV_{\tGL_2}^{k,0} \otimes \Big( \rV_{\tGL(2)}^{c(\lambda),c(\lambda)} \otimes V_{\Spin(2(n-2))}^{\lambda'} \Big) \cong
\rV_{\tGL(2)}^{k + c(\lambda),c(\lambda)} \otimes \rV_{\Spin(2(n-2))}^{\lambda'},
\end{equation*}
and the right side corresponds to $\rV_\rL^{k\omega_1 + \lambda}$.
This proves~\eqref{tp:a1-la}.

Similarly, $\rV_\rL^{\omega_3 - \omega_2}$ corresponds to $\rV_{\Spin(2(n-2))}^{\omega'_3}$
(which, according to our conventions, is the first fundamental representation of~$\Spin(2(n-2))$).
By~\cite[Table~5]{VO} its middle exterior power
is isomorphic to the direct sum of two irreducible representations with highest weights $2\omega'_n$ and $2\omega'_{n-1}$.
By~\eqref{eq:omega-prime} these weights correspond to the weights $2\omega_n - \omega_2$ and $2\omega_{n-1} - \omega_2$, and the claim follows.
This proves~\eqref{eq:wedge-n-2}.
\end{proof}

Finally, we deduce from Borel--Bott--Weil a vanishing lemma,
on which most of semiorthogonality results for~$\OG(2,2n)$ in the next section rely.

\begin{lemma}
\label{lemma:acyclicity-new}
Let $\lambda = (\lambda_1,\dots,\lambda_n)$ be a dominant weight of $\rL$.
Assume that one of the following two conditions holds
\begin{align}
\label{acyclicity:1}
|\lambda_i + n - i| & < n - 1	&\text{for all $i$}, \\
\intertext{or}
\label{acyclicity:2}
|\lambda_i + n - i| & < n - 2 	&\text{for all but one $i$}.
\end{align}
Then $\cU^\lambda$ is acyclic.
\end{lemma}

\begin{proof}
These are special cases of~\cite[Lemma~5.2]{KuPo}; we provide a proof for completeness.

In view of~\eqref{eq:rho-dn}, if~\eqref{acyclicity:1} holds,
then all coordinates of $\lambda + \rho$ have absolute values less than~\mbox{$n - 1$}.
Therefore, at least two of them have the same absolute values.
Similarly, if~\eqref{acyclicity:2} holds, then all but one coordinates of $\lambda + \rho$ have absolute values less than~\mbox{$n - 2$}.
Therefore, again at least two of them have the same absolute values.
In both cases we see that~\mbox{$\lambda + \rho$} lies on one of the walls~\eqref{eq:walls-dn},
and by Borel--Bott--Weil theorem we conclude that~$\cU^\lambda$ is acyclic.
\end{proof}

\subsection{Exceptional collection}
\label{subsection:exceptional-dn}

In this section we prove that~\eqref{eq-collection-OG-even-I} is an exceptional collection. Recall from~\eqref{eq:bundles-dn} that $\cU$ denotes the dual bundle of $\cU^{\omega_1}$. Consequently, $S^k\cU^\vee \cong \cU^{k\omega_1}$.

First, we discuss the ``tautological'' part of the collection. Recall that an exceptional collection $E_1,\dots,E_k$ of vector bundles is {\sf strong},
if $\Ext^{>0}(E_i,E_j) = 0$ for all $1 \le i,j \le k$.

\begin{lemma}
\label{lemma:so1}
The collection of vector bundles $\cO,\cU^\vee,\dots,S^{n-2}\cU^\vee$ is a strong exceptional collection. Moreover, for $0 \le k, l \le n-2$ and $1 \le t \le 2n - 4$ we have
\begin{equation*}
\Ext^\bullet(S^k\cU^\vee(t), S^l\cU^\vee) =
\begin{cases}
\Bbbk[4 - 2n], 	& \text{if $k = l = n - 2$, $t = n - 2$,}\\
\Bbbk[3 - 2n], 	& \text{if $k = l = n - 2$, $t = n - 1$,}\\
0,		& \text{otherwise.}
\end{cases}
\end{equation*}
\end{lemma}

\begin{proof}
We will show that $\Ext^\bullet(S^k\cU^\vee(t),S^l\cU^\vee)$ is zero for $0 \le k,l \le n-2$ and $0 \le t \le 2n-4$,
except for $t = 0$, $k \le l$ and $t \in \{n - 2, n - 1\}$, $k = l = n-2$.
By~\eqref{eq:omega-dn} and Serre duality we can assume $0 \le t \le n - 2$.

By Corollary~\ref{corollary:duals} and~\eqref{tp:a1*b1} we have
\begin{multline*}
(S^k\cU^\vee(t))^\vee \otimes S^l\cU^\vee \cong S^k\cU^\vee(-k-t) \otimes S^l\cU^\vee \\ \cong
\bigoplus_{j = 0}^{\min\{k,l\}} S^{k + l - 2j}\cU^\vee(j - k - t) =
\bigoplus_{j = 0}^{\min\{k,l\}} \cU^{(l - j - t,j - k - t,0,\dots,0)}.
\end{multline*}
So, it is enough to compute the cohomology of bundles $\cU^{\lambda}$
with $\lambda = (l - j - t,j - k - t,0,\dots,0)$, so that
\begin{equation*}
2 - n \le \lambda_1 \le n - 2,\quad
4 - 2n \le \lambda_2 \le 0,\qquad
\lambda_3 = \dots = \lambda_n = 0.
\end{equation*}
It is convenient to distinguish three cases:
\begin{description}
  \item[Case 1] If $4 - 2n < \lambda_2 < 0$, then $|\lambda_i + n - i| < n - 2$ for all $i \ge 2$.
  Hence,~\eqref{acyclicity:2} holds and~$\cU^\lambda$ is acyclic by Lemma~\ref{lemma:acyclicity-new}.

  \item[Case 2] If $\lambda_2 = 0$, then we necessarily have $t = 0$ and $j = k$.
  Further, since $j \le l$, we obtain~$k \le l$. In particular, the weight $\lambda$ is dominant, hence $H^{>0}(\OG(2,2n),\cU^\lambda) = 0$.
  Moreover, if $k = l$, then $\lambda = 0$ and in this case the cohomology of~$\cU^\lambda$ is isomorphic to $\Bbbk[0]$.
\end{description}
A combination of Case 1 and Case 2 proves that $\cO, \cU^\vee, \dots, S^{n-2}\cU^\vee$ is a strong exceptional collection.
\begin{description}
  \item[Case 3] If $\lambda_2 = 4 - 2n$, then we necessarily have $k = t = n - 2$ and $j = 0$.
  In this case we have $2 - n \le \lambda_1 \le 0$.

  If $\lambda_1 < 0$, then $|\lambda_i + n - i| < n - 1$ for all $i$. Hence~\eqref{acyclicity:1} holds and $\cU^\lambda$ is acyclic by Lemma~\ref{lemma:acyclicity-new}.

  If $\lambda_1 = 0$, then $l = n - 2$, and we have $\cU^\lambda = \cU^{(0,4-2n,0,\dots,0)}$.
  In this case the Borel--Bott--Weil theorem implies that the cohomology of $\cU^\lambda$ is isomorphic to $\Bbbk[4 - 2n]$.
  Indeed, $\lambda + \rho = (n-1,2-n,n-3,\dots,0) = w\rho$, where
  \begin{equation*}
  w = (\rs_2 \cdot \ldots \cdot \rs_{n-2}) \cdot \rs_{n-1} \cdot \rs_n \cdot (\rs_{n-2} \cdot \ldots \cdot \rs_2)
  \end{equation*}
  and $\rs_i$ are the simple reflections.
\end{description}
This completes the proof of the lemma.
\end{proof}

Next, we discuss the ``spinor'' part of the collection.

\begin{lemma}
\label{lemma:so2}
The collection of sheaves $\cU^{\omega_{n-1}}, \cU^{\omega_{n}}, \cU^{2\omega_{n-1}}, \cU^{2\omega_{n}},
\{ \cU^{\omega_{n-1}}(t), \cU^{\omega_{n}}(t) \}_{1 \le t \le 2n - 4}$
is exceptional.
\end{lemma}
\begin{proof}
We will compute $\Ext^\bullet(\cU^{a\omega_i}(t),\cU^{b\omega_j})$ where $i,j \in \{n - 1,n\}$, and
\begin{itemize}
\item
either \mbox{$a,b \in \{1,2\}$} and $(a,b) \ne (2,2)$, and~\mbox{$1 \le t \le 2n - 4$},
\item
or \mbox{$a,b \in \{1,2\}$} and $(a,b) \ne (1,2)$, and $t = 0$.
\end{itemize}
By~\eqref{eq:omega-dn} and Serre duality we can assume $0 \le t \le n - 2$.

We take $\lambda = a\omega_{i} + t\omega_2 = (t+\tfrac{a}2,t+\tfrac{a}2,\tfrac{a}2,\dots,\tfrac{a}2,\pm\tfrac{a}2)$
and $\mu = b\omega_{j} = (\tfrac{b}2,\tfrac{b}2,\tfrac{b}2,\dots,\tfrac{b}2,\pm \tfrac{b}2)$;
the signs of the last coordinate depend on whether~$i$ and~$j$ are~$n-1$ or~$n$.
Since the~$\rW_\rL$-orbit of~$\lambda$ consists of points $(t+\tfrac{a}2,t+\tfrac{a}2,\pm\tfrac{a}2,\dots,\pm\tfrac{a}2,\pm\tfrac{a}2)$
(with the parity of the number of negative signs depending on~$i$),
by Lemma~\ref{lemma:prv} we have
\begin{equation}
\label{eq:ulm}
(\cU^\lambda)^\vee \otimes \cU^\mu = \bigoplus_{\nu} (\cU^\nu)^{\oplus m(\lambda,\mu,\nu)},
\end{equation}
where $\nu$ runs over the integral or half-integral points satisfying
\begin{equation*}
\nu_1 = \nu_2 = \tfrac{b-a}2 - t,
\qquad \nu_3,\dots,\nu_{n-1} = \tfrac{b \pm a}2 \in [-\tfrac12,2],
\qquad \nu_n = \tfrac{\pm b \pm a}2 \in [-2,2].
\end{equation*}
(in fact we could further restrict the combinations of signs that appear in~$\nu$,
but since we are interested in an upper bound on direct summands of~$(\cU^\lambda)^\vee \otimes \cU^\mu$, we can neglect this).
It is convenient to distinguish three cases:
\begin{description}
  \item[Case 1]
  If $1 \le t \le n-2$ and~$(a,b) \ne (2,2)$, then
  \begin{equation*}
    \nu_1 = \nu_2 \in [\tfrac32 - n, -\tfrac12], \qquad \nu_3,\dots,\nu_{n-1},\nu_n \in [-\tfrac32,\tfrac32],
  \end{equation*}
  For any such point we have $|\nu_i + n - i| < n - 1$ for all $i$, hence~\eqref{acyclicity:1} holds for~$\nu$ and,
  therefore, all summands in~\eqref{eq:ulm} are acyclic by Lemma~\ref{lemma:acyclicity-new}.

  \item[Case 2]
  If $t = 0$ and $a > b$, then
  \begin{equation*}
    \nu_1 = \nu_2 = -\tfrac12, \qquad \nu_3,\dots,\nu_{n-1} \in [-\tfrac12,\tfrac32], \qquad \nu_n \in [-\tfrac32, \tfrac32].
  \end{equation*}
  For any such point we still have $|\nu_i + n - i| < n - 1$ for all $i$, hence again all summands in~\eqref{eq:ulm} are acyclic.

  \item[Case 3]
  If $t = 0$ and $a = b$, then
  \begin{equation*}
    \nu_1 = \nu_2 = 0, \qquad \nu_3,\dots,\nu_{n-1} \in [0,2], \qquad \nu_n \in [-2, 2].
  \end{equation*}
  It is easy to see that the only non-acyclic bundle among $\cU^\nu$ is the trivial bundle.
  Indeed, if $\tilde\nu = \nu + \rho$ then $(\tilde\nu_1, \tilde\nu_2) = (n-1,n-2)$
  and $n - 1 \ge \tilde\nu_3 \ge \dots \ge \tilde\nu_{n-1} \ge 1$,
  and the only possibility to satisfy these restriction and avoid the walls~\eqref{eq:walls-dn}
  is to have $\tilde\nu_i = n - i$ for~$1 \le i \le n - 1$.
  Finally, $0 \le \tilde\nu_n \le 2$, and sine $n \ge 3$, it must be zero to avoid the walls.
  By Lemma~\ref{lemma:prv} the trivial bundle is a summand of $(\cU^\lambda)^\vee \otimes \cU^\mu$
  only when $\lambda = \mu$ and then its multiplicity is equal to~1.
  Hence, in this case we have~$\Ext^\bullet(\cU^\lambda,\cU^\mu) = \Bbbk[0]$.
\end{description}
This completes the proof of the lemma.
\end{proof}

In the next two lemmas we show that the ``tautological'' part of the collection can be merged with the ``spinor'' part.

\begin{lemma}
\label{lemma:so3}
We have $\Ext^\bullet(\cU^{\omega_j}(t),S^k\cU^\vee) = 0$ for all $0 \le k \le n - 2$, $j \in \{n - 1,n\}$, and~$0 \le t \le 2n - 4$.
\end{lemma}
\begin{proof}
By Corollary~\ref{corollary:duals} and~\eqref{tp:a1-la} we have
\begin{equation*}
(\cU^{\omega_j}(t))^\vee \otimes S^k\cU^\vee \cong
\cU^{\omega_{j'} - (t+1)\omega_2} \otimes \cU^{k \omega_1} \cong
\cU^{k\omega_1 - (t+1)\omega_2 + \omega_{j'}},
\end{equation*}
where $j' = j$ if $n$ is even and $j' = 2n - 1 - j$ if $n$ is odd.
So, it is enough to compute the cohomology of bundles $\cU^\lambda$ with
\begin{equation*}
\lambda = \left(k - t - \tfrac12, - t - \tfrac12, \tfrac12,\dots,\tfrac12,\pm\tfrac12 \right).
\end{equation*}
If $0 \le t \le 2n - 5$, then $|\lambda_i + n - i| < n - 2$ for $i \ge 2$.
Hence~\eqref{acyclicity:2} holds and the bundle $\cU^\lambda$ is acyclic by Lemma~\ref{lemma:acyclicity-new}.
Similarly, if~$t = 2n - 4$, then we have $|\lambda_i + n - i| < n - 1$ for all~$i$.
Hence~\eqref{acyclicity:1} holds and the bundle $\cU^\lambda$ is acyclic by Lemma~\ref{lemma:acyclicity-new}.
\end{proof}

\begin{lemma}
\label{lemma:so4}
If $0 \le k \le n - 2$, $j \in \{n - 1,n\}$, and~$1 \le t \le 2n - 2$, then
\begin{equation*}
\Ext^\bullet(S^k\cU^\vee(t),\cU^{2\omega_j}) =
\begin{cases}
\Bbbk[5 - 3n], 	& \text{if $k = n - 2$ and $t = n$,}\\
\Bbbk[2 - n], 	& \text{if $k = n - 2$ and $t = 1$,}\\
0,		& \text{otherwise.}
\end{cases}
\end{equation*}
\end{lemma}

In fact, using isomorphisms of Corollary~\ref{corollary:duals} and Serre duality
it is possible to identify one non-trivial $\Ext$-space above with the dual of the other.

\begin{proof}
As before, we have
\begin{equation*}
(S^k\cU^\vee(t))^\vee \otimes \cU^{2\omega_j} \cong
S^k\cU^\vee(-k-t) \otimes \cU^{2\omega_j} \cong
\cU^{k \omega_1 - (k + t)\omega_2 + 2\omega_j}.
\end{equation*}
So, it is enough to compute the cohomology of the bundle $\cU^\lambda$ with
\begin{equation*}
\lambda = (1-t,1-k-t,1,\dots,1,\pm1).
\end{equation*}
For this we directly apply the Borel--Bott--Weil theorem.
We have
\begin{equation*}
\lambda + \rho = (n - t, n - 1 - k - t, n - 2, \dots, 2, \pm 1).
\end{equation*}
It is convenient to distinguish several cases:
\begin{description}
\item[Case 1]
If $2 \le t \le 2n - 2$ and $t \ne n$, then the absolute value of the first coordinate of~\mbox{$\lambda + \rho$}
is equal to the absolute value of one of the last $n - 2$ coordinates.
Hence, the bundle~$\cU^\lambda$ is acyclic.
\item[Case 2]
If $t = n$ and $0 \le k \le n - 3$, then $n - 1 - k - t = - k - 1$,
and the absolute value of the second coordinate of $\lambda + \rho$ is equal to the absolute value of one of the last $n - 2$ coordinates.
Hence, $\cU^\lambda$ is acyclic.
\item[Case 3]
If $t = n$ and~$k = n - 2$, then $\lambda + \rho = (0, 1 - n, n - 2, \dots, 2, \pm 1) = w\rho$, where
\begin{equation*}
w = (\rs_{n-1} \cdot \rs_{n-2} \cdot \ldots \cdot \rs_2) \cdot \rs_1 \cdot
(\rs_2 \cdot \ldots \cdot \rs_{n-2}) \cdot \rs_{n-1} \cdot \rs_n \cdot (\rs_{n-2} \cdot \ldots \cdot \rs_2)
\end{equation*}
and $\rs_i$ are the simple reflections.
By the Borel--Bott--Weil theorem the cohomology of $\cU^\lambda$ equals $\Bbbk[5 - 3n]$.
\item[Case 4]
If $t = 1$ and $0 \le k \le n - 3$, then $n - 1 - k - t = n - 2 - k$,
and the absolute value of the second coordinate of $\lambda + \rho$ is equal to the absolute value of one of the last $n - 2$ coordinates.
Hence, $\cU^\lambda$ is acyclic.
\item[Case 5]
If $t = 1$ and $k = n - 2$, then $\lambda + \rho = (n - 1, 0, n - 2, \dots, 2, \pm 1) = w\rho$, where
\begin{equation*}
w = \rs_{n-1} \cdot \rs_{n-2} \cdot \ldots \cdot \rs_2
\end{equation*}
and $\rs_i$ are the simple reflections.
By the Borel--Bott--Weil theorem the cohomology of $\cU^\lambda$ equals $\Bbbk[2 - n]$.
\end{description}
This completes the proof of the lemma.
\end{proof}

Now we deduce

\begin{proposition}
The collection~\eqref{eq-collection-OG-even-I} is exceptional.
\end{proposition}
\begin{proof}
A combination of Lemmas~\ref{lemma:so1}, \ref{lemma:so2}, \ref{lemma:so3}, \ref{lemma:so4}, and Serre duality proves the proposition.
\end{proof}

\subsection{Some useful complexes}

In this section we construct some exact sequences and complexes that will be used to check that the collection~\eqref{eq-collection-OG-even-I} is full.

Let $V$ be a vector space of dimension~$2n \ge 8$ endowed with a non-degenerate symmetric bilinear form
and let $\OG(2,V) = \OG(2,2n)$ be the corresponding even orthogonal Grassmannian.
We denote $\cV = V \otimes \cO$ the trivial vector bundle with fiber~$V$.
Note that~$V$ and~$\cV$ are canonically self-dual (via the chosen symmetric bilinear form on~$V$).
We have the tautological short exact sequence on $\OG(2,2n)$
\begin{equation*}
0 \to \cU \to \cV \to \cV/\cU \to 0
\end{equation*}
and its dual
\begin{equation*}
0 \to \cU^\perp \to \cV \to \cU^\vee \to 0.
\end{equation*}
Taking exterior powers of the above sequences, we obtain the long exact sequences
\begin{multline}
\label{eq-useful-complex}
0 \to S^m\cU \to V \otimes S^{m-1} \cU \to \Lambda^2 V \otimes S^{m-2} \cU \to \dots \\
\dots  \to \Lambda^{m-1} V \otimes \cU \to \Lambda^m V \otimes \cO \to \Lambda^m(\cV/\cU) \to 0
\end{multline}
and
\begin{multline}
\label{eq-useful-complex-2}
0 \to \Lambda^m\cU^\perp \to \Lambda^m V \otimes \cO \to \Lambda^{m-1}V \otimes \cU^\vee \to \dots \\
\dots \to \Lambda^2V \otimes S^{m-2}\cU^\vee \to V \otimes S^{m-1}\cU^\vee \to S^m\cU^\vee \to 0
\end{multline}

Recall notation~\eqref{eq:bundles-dn} for spinor bundles.
We denote their spaces of global sections by
\begin{equation*}
\bfS_- := H^0(\OG(2,V),\cS_-) = \rV_{\Spin(2n)}^{\omega_{n-1}},\qquad
\bfS_+ := H^0(\OG(2,V),\cS_+) = \rV_{\Spin(2n)}^{\omega_n}
\end{equation*}
(the {\sf half-spinor representations} of~$\Spin(2n)$).
We set
\begin{equation}
\label{eq:wi}
W_i = \bigoplus_{s = 0}^{\lfloor i/2 \rfloor} \Lambda^{i - 2s} V.
\end{equation}
We will combine~\eqref{eq-useful-complex} and~\eqref{eq-useful-complex-2} to prove the following

\begin{proposition}
On $\OG(2,2n)$ there exists an exact sequence
\begin{multline}
\label{eq-OG-2-2n-exact-sequence-for-Sym}
0 \to S^{n-2} \cU \to  W_1 \otimes S^{n-3} \cU \to W_2 \otimes S^{n-4} \cU \to \dots
\\
\dots \to W_{n-2} \otimes \cO
\to \bfS_+ \otimes \cS_+ \oplus \bfS_- \otimes \cS_-   \to
W_{n-2} \otimes \cO(1)  \to \dots
\\
\dots \to W_2 \otimes S^{n-4} \cU^\vee(1) \to W_1 \otimes S^{n-3} \cU^\vee(1) \to S^{n-2} \cU^\vee(1) \to 0.
\end{multline}
\end{proposition}

\begin{proof}
If $n$ is even, we apply~\cite[Proposition 6.7]{Ku08a} and conclude that there is a filtration on~\mbox{$\bfS_+ \otimes \cS_+^\vee \oplus \bfS_- \otimes \cS_-^\vee$} with factors of the form $\Lambda^{2s}\cU^\perp$, where
$0 \le s \le n - 1$. Tensoring by~$\cO(1)$ and taking into account the isomorphisms
\begin{equation*}
\cS_+^\vee(1) \cong \cS_+,
\qquad
\cS_-^\vee(1) \cong \cS_-
\end{equation*}
(Corollary~\ref{corollary:duals}) we obtain a filtration on $\bfS_+ \otimes \cS_+ \oplus \bfS_- \otimes \cS_-$
with factors of the form $\Lambda^{2s}\cU^\perp(1)$, where $0 \le s \le n - 1$.

Similarly, if $n$ is odd, the argument of~\cite[Proposition 6.7]{Ku08a} proves that
there is a filtration on $\bfS_+ \otimes \cS_-^\vee \oplus \bfS_- \otimes \cS_+^\vee$
with factors of the form $\Lambda^{2s+1}\cU^\perp$, where $0 \le s \le n - 2$.
Tensoring by~$\cO(1)$ and taking into account isomorphisms
\begin{equation*}
\cS_+^\vee(1) \cong \cS_-,
\qquad
\cS_-^\vee(1) \cong \cS_+,
\end{equation*}
we finally obtain a filtration on $\bfS_+ \otimes \cS_+ \oplus \bfS_- \otimes \cS_-$
with factors of the form $\Lambda^{2s+1}\cU^\perp(1)$, where~\mbox{$0 \le s \le n - 2$}.

Thus, in both cases we have a filtration on $\bfS_+ \otimes \cS_+ \oplus \bfS_- \otimes \cS_-$
with factors
\begin{equation*}
\Lambda^{2n - 2 - r}\cU^\perp(1),
\Lambda^{2n - 4 - r}\cU^\perp(1), \dots,
\Lambda^{n}\cU^\perp(1),
\Lambda^{n - 2}\cU^\perp(1), \dots,
\Lambda^{2 + r}\cU^\perp(1),
\Lambda^{r}\cU^\perp(1),
\end{equation*}
where $r = 0$ if $n$ is even and $r = 1$ if $n$ is odd.
Note that the total number of factors $n - r$ is even in both cases.
This means that there is an exact sequence of vector bundles
\begin{equation*}
0 \to F_- \to \bfS_+ \otimes \cS_+ \oplus \bfS_- \otimes \cS_- \to F_+ \to 0,
\end{equation*}
where $F_-$ has a filtration with factors $\Lambda^{2n - 2 - r}\cU^\perp(1), \Lambda^{2n - 4 - r}\cU^\perp(1), \dots, \Lambda^{n}\cU^\perp(1)$
and $F_+$ has a filtration with factors $\Lambda^{n - 2}\cU^\perp(1), \dots, \Lambda^{2 + r}\cU^\perp(1), \Lambda^{r}\cU^\perp(1)$.
So, it is enough to show that~$F_-$ has a left resolution given by the first half of~\eqref{eq-OG-2-2n-exact-sequence-for-Sym},
and that~$F_+$ has a right resolution given by the second half of~\eqref{eq-OG-2-2n-exact-sequence-for-Sym}.

Indeed, each factor of the filtration of $F_-$ can be rewritten as
\begin{equation*}
\Lambda^{n + 2k}\cU^\perp(1) \cong \Lambda^{n - 2 - 2k}(\cV/\cU),
\end{equation*}
where $0 \le k \le (n - 2 - r)/2$,
hence has the resolution~\eqref{eq-useful-complex} with $m = n - 2 - 2k$.
Since by Lemma~\ref{lemma:so1} the exceptional collection $\{ S^i\cU \}_{i = 0}^{n-2}$ is \emph{strong},
the extensions of the filtration factors $\Lambda^{n + 2k}\cU^\perp(1)$ in $F_-$ can be realized
by a bicomplex with rows given by~\eqref{eq-useful-complex} with~\mbox{$m = n - 2 - 2k$} shifted by $-k$.
The totalization of this bicomplex can be written as
\begin{equation*}
S^{n-2} \cU \to  W_1 \otimes S^{n-3} \cU \to W_2 \otimes S^{n-4} \cU \to \dots \to W_{n-2} \otimes \cO
\end{equation*}
and thus provides a left resolution for $F_-$.

Similarly, each factor $\Lambda^{n-2-2k}\cU^\perp(1)$ of the filtration of $F_+$ has resolution~\eqref{eq-useful-complex-2}
twisted by~$\cO(1)$ (where $m = n - 2 - 2k$), and the above argument shows that $F_+$ has the right resolution
\begin{equation*}
W_{n-2} \otimes \cO(1)  \to \dots \to W_2 \otimes S^{n-4} \cU^\vee(1) \to W_1 \otimes S^{n-3} \cU^\vee(1) \to S^{n-2} \cU^\vee(1).
\end{equation*}
This completes the proof of the proposition.
\end{proof}

Note that the composition of morphisms $\cU \to \cV \to \cU^\vee$ is zero, hence $\cU \subset \cU^\perp$.
Moreover,
\begin{equation}
\label{eq:uperp-u}
\cU^\perp/\cU \cong \cU^{\omega_3 - \omega_2}.
\end{equation}

We will need the following lemma.

\begin{lemma}\label{lemma-double-complex}
For any $m \le 2n - 4$ on $\OG(2,2n)$ there exists a double complex
\begin{equation*}
\begin{small}
\xymatrix@C=1em@R=3ex{
\Lambda^m V \otimes \cO  \ar[r] &  \Lambda^{m-1} V \otimes \cU^\vee \ar[r] &  \Lambda^{m-2} V \otimes S^2 \cU^\vee \ar[r] & \dots \ar[r] &  V \otimes S^{m-1} \cU^\vee    \ar[r] &       S^m \cU^\vee      \\
\Lambda^{m-1} V \otimes \cU     \ar[r] \ar[u] &  \Lambda^{m-2} V \otimes \cU  \otimes \cU^\vee \ar[r] \ar[u] &  \Lambda^{m-3} V \otimes \cU  \otimes S^2 \cU^\vee  \ar[r] \ar[u] & \dots \ar[r]  &   \cU \otimes S^{m-1} \cU^\vee  \ar[u]          \\
\dots                                  \ar[u] &  \dots \ar[u] &  \dots \ar[u] &  \raisebox{-1.2ex}{\reflectbox{$\ddots$}} \\
\Lambda^2 V \otimes S^{m-2} \cU           \ar[r] \ar[u] &   V \otimes S^{m-2} \cU \otimes \cU^\vee \ar[r]\ar[u] &   S^{m-2} \cU \otimes S^2\cU^\vee     \ar[u]    \\
V \otimes S^{m-1} \cU           \ar[r] \ar[u] &   S^{m-1} \cU \otimes \cU^\vee   \ar[u]      \\
S^m \cU  \ar[u]
}
\end{small}
\end{equation*}
whose total complex has only one non-trivial cohomology in the middle term isomorphic to~$\Lambda^m (\cU^\perp/\cU)$.
\end{lemma}

\begin{proof}
The bicomplex is just the $m$-th exterior power of the complex
\begin{equation*}
\cU \to \cV \to \cU^\vee,
\end{equation*}
which is quasiisomorphic to $\cU^\perp/\cU$, hence the claim (cf.\ the argument of~\cite[Lemma~7.3]{Ku08a}).
\end{proof}

Finally, we will need the following lemma from~\cite{Ku08a}.

\begin{lemma}[\cite{Ku08a}, Proposition 6.3]
\label{lemma:sss}
There exists a complex
\begin{equation*}
0 \to \cS_- \to \bfS_- \otimes \cO(1) \to \cS_-(1) \to 0,
\end{equation*}
whose only non-trivial cohomology group is in the middle term and is isomorphic to $\cS_+ \otimes \cU^\vee$.
\end{lemma}

\subsection{Fullness}
\label{subsection:fullness-dn}

The proof of the fullness of the collection~\eqref{eq-collection-OG-even-I} goes
via restriction to odd Grassmannians $\OG(2,2n-1)$, and is similar to the proofs of fullness given in~\cite{Ku08a}.
For technical reasons it is more convenient to work with the collection
\begin{align}
\label{eq-collection-OG-even-II}
\cA(2-n) , \dots ,\cA(-1), \cU^{2\omega_{n-1}}(-1), \cU^{2\omega_{n}}(-1) , \cA , \cB(1) , \dots , \cB(n-2).
\end{align}
By Serre duality, exceptionality of~\eqref{eq-collection-OG-even-I} implies exceptionality of~\eqref{eq-collection-OG-even-II}.
Conversely, fullness of~\eqref{eq-collection-OG-even-II} implies fullness of~\eqref{eq-collection-OG-even-I}.

We will deduce fullness of~\eqref{eq-collection-OG-even-II} from the following key lemma.

\begin{lemma}\label{key-lemma}
Let $\E$ be the set of bundles appearing in~\eqref{eq-collection-OG-even-II}
and let $\E'$ be its subset defined as
\begin{equation}
\label{eq:eprime}
\E' := \{ S^k\cU^\vee(t), \cS_+(t) \mid 0 \le k \le n-3,\ 2 - n \le t \le n - 3 \} \subset \E.
\end{equation}
Then for any object $E \in \E\,' \,$ we have
\begin{align}
\label{kl:twist}
E \otimes \cO(1)   & \in \langle \E \rangle,
\\
\label{kl:tensor}
E \otimes \cU^\vee & \in \langle \E \rangle.
\end{align}
\end{lemma}

\begin{proof}
The inclusion~\eqref{kl:twist}
follows immediately
from comparison of~\eqref{eq:eprime}, \eqref{eq:a-b-dn} and~\eqref{eq-collection-OG-even-II}.
Thus, we concentrate here on the inclusion~\eqref{kl:tensor}.

We split objects $E \in \E'$ into two classes:
\begin{itemize}
\item $\cS_+(t)$ with $t \in [2-n,n-3]$ (spinor bundles), and
\item $S^k \cU^\vee(t)$ with $k \in [0,n-3]$ and $t \in [2-n,n-3]$ (tautological bundles),
\end{itemize}
and treat these separately.

The class of spinor bundles is easy: from Lemma~\ref{lemma:sss} we conclude
\begin{equation*}
\cS_+(i) \otimes \cU^\vee \in \langle \cS_-(i), \cO(i+1), \cS_-(i+1) \rangle.
\end{equation*}
and the inclusion~\eqref{kl:tensor} for $E = {} \cS_+(i)$ follows.

The class of tautological bundles is a bit more tedious.
First, note that by~\eqref{tp:a1*b1} we have
\begin{equation*}
S^k \cU^\vee \otimes \cU^\vee = S^{k+1} \cU^\vee \oplus S^{k-1}\cU^\vee(1).
\end{equation*}
Thus, the inclusion~\eqref{kl:tensor} would follow from the inclusions
\begin{equation}
\label{eq:inclusions}
S^j\cU^\vee(t) \in \langle \E \rangle,
\qquad 0 \le j \le n -2,\quad 2 - n \le t \le n - 2.
\end{equation}
So, proving these inclusions, we will prove the lemma.

For $j \le n - 3$ and for $j = n - 2$, $1 \le t \le n - 2$ the inclusions~\eqref{eq:inclusions}
follow immediately from~\eqref{eq-collection-OG-even-II}. Furthermore, using~\eqref{eq-OG-2-2n-exact-sequence-for-Sym} twisted by $\cO(n - 2 + t)$ and isomorphisms
\begin{equation*}
S^l \cU \simeq S^l \cU^\vee(-l),
\end{equation*}
we deduce~\eqref{eq:inclusions} for $j = n - 2$ and $2 - n \le t \le -1$.

So, it only remains to show that
\begin{equation*}
S^{n-2}\cU^\vee \in  \langle \E \rangle.
\end{equation*}
For this we consider the double complex of Lemma~\ref{lemma-double-complex} with $m = n-2$.
Using the direct sum decompositions
\begin{equation*}
S^k\cU^\vee \otimes S^l\cU \cong \bigoplus_{j = 0}^{\min\{k,l\}} S^{k + l - 2j}\cU^\vee(j - l)
\end{equation*}
and the cases of the inclusion~\eqref{eq:inclusions} proved above,
we see that all the terms of the double complex with a possible exception for the rightmost term $S^{n-2}\cU^\vee$
are contained in the subcategory generated by~$\E$.
By~\eqref{eq:wedge-n-2} and~\eqref{eq:uperp-u} the unique cohomology sheaf
\begin{equation*}
\cU^{2\omega_{n-1}}(-1) \oplus \cU^{2\omega_n}(-1)
\end{equation*}
of the double complex is contained in $\langle \E \rangle$,
therefore the last term $S^{n-2}\cU^\vee$ is in the subcategory generated by~$\E$.
\end{proof}

Let
\begin{equation*}
v \in H^0(\OG(2,2n), \cU^\vee) = V^\vee \cong V
\end{equation*}
be any non-zero global section and let $v^\perp \subset V$ be the orthogonal complement of $v$.
The restriction of the bilinear form to $v^\perp$ is nondegenerate if and only if $v$ is non-isotropic.
In this case the zero locus of $v$ (considered as a section of~$\cU^\vee$)
is the odd orthogonal Grassmannian~$\OG(2,v^\perp) {} \cong \OG(2,2n-1)$ and we have the natural closed embedding
\begin{equation*}
i_v \colon \OG(2,v^\perp) \to \OG(2,V)
\end{equation*}
and the Koszul resolution
\begin{equation}
\label{eq-koszul-resolution}
0 \to \cO(-1) \to \cU \to \cO \to i_{v*} \cO_{\OG(2,v^\perp)} \to 0.
\end{equation}
The union of $\OG(2,v^\perp)$ for non-isotropic $v \in V$ sweeps~$\OG(2,V)$, hence we have the following

\begin{lemma}[{\cite[Lemma 4.5]{Ku08a}}]
\label{lemma:restriction}
If for an object $F \in \Db(\OG(2,2n))$ the restrictions $i_v^* F$ vanish for all non-isotropic $v \in V$, then $F = 0$.
\end{lemma}

Now we are finally ready to prove the fullness of~\eqref{eq-collection-OG-even-II}.

\begin{proposition}
\label{lemma-vanishing-of-the-restriction}
If $F \in \Db(\OG(2,V))$ is right orthogonal to all the vector bundles $E$ in the collection~\eqref{eq-collection-OG-even-II},
i.e., $\Ext^\bullet(E,F) = 0$, then $F = 0$.
\end{proposition}

\begin{proof}
The assumption of the proposition can be rewritten as
\begin{equation}
\label{eq-orthogonality-to-all-objects-in-E}
\Ext^\bullet(E,F) = H^{\bullet}(\OG(2,V),E^{\vee} \otimes F)=0 \qquad \forall E \in \E.
\end{equation}
Now take any non-isotropic vector~$v$ and any bundle $E \in \E'$, and tensor \eqref{eq-koszul-resolution} by $E^{\vee} \otimes F$.
We obtain an exact sequence
\begin{equation}
\label{eq-koszul-resolution-tensored}
0 \to \cO(-1) \otimes E^{\vee} \otimes F \to \cU \otimes E^{\vee} \otimes F \to E^{\vee} \otimes F \to i_{v*} i_v^* (E^{\vee} \otimes F) \to 0
\end{equation}
From Lemma~\ref{key-lemma} and~\eqref{eq-orthogonality-to-all-objects-in-E}
it follows that the cohomology groups of the first three terms of this complex vanish.
Hence, also the cohomology of $i_{v*} i_v^* (E^{\vee} \otimes F)$ has to vanish, and we conclude that
\begin{equation}
\label{eq:vanishing}
H^{\bullet} (\OG(2,V), i_{v*} i_v^* (E^{\vee} \otimes F)) =
H^{\bullet} (\OG(2,v^\perp), i_v^* (E^{\vee} \otimes F)) = \Ext^\bullet(i_v^* E, i_v^* F) = 0
\end{equation}
for any $E \in \E'$.
In other words, the object $i_v^*F$ is orthogonal to all objects $i_v^*E$ for $E \in \E'$.

Set
\begin{equation*}
\cC := \langle \cO, i_v^*(\cU^\vee), \dots, i_v^*(S^{n-3}\cU^\vee), i_v^*(\cS_+) \rangle \subset \Db(\OG(2,v^\perp)).
\end{equation*}
Since $i_v^*(\cU^\vee)$ is the dual tautological bundle and $i_v^*(\cS_+)$ is the spinor bundle
on the odd orthogonal Grassmannian~$\OG(2,v^\perp) \cong \OG(2,2n-1)$,
it follows from~\cite[Theorem 7.1]{Ku08a} that there is a semiorthogonal decomposition
\begin{equation}
\label{eq-collection-OG-odd-I}
\Db(\OG(2,2n-1)) = \langle \cC , \cC(1) , \dots , \cC(2n-5) \rangle.
\end{equation}
Twisting it by~$\cO(2 - n)$ we obtain a semiorthogonal decomposition
\begin{equation}
\label{eq-collection-OG-odd-II}
\Db(\OG(2,2n-1)) = \langle \cC(2-n), \dots, \cC(-1), \cC , \cC(1) , \dots , \cC(n-3) \rangle.
\end{equation}
Thus, the definition of the set~$\E'$ implies that the objects $i_v^*E$ for $E \in \E'$ generate the category~$\Db(\OG(2,v^\perp))$,
hence~\eqref{eq:vanishing} implies that $i_v^*F = 0$.
Since this holds for any non-isotropic~$v$, it follows from Lemma~\ref{lemma:restriction} that $F = 0$.
\end{proof}

\subsection{Residual category}
\label{subsection:residual-dn}

For $1 \le i \le n - 2$ we define $F_i$ to be the stupid right truncation of complex~\eqref{eq-OG-2-2n-exact-sequence-for-Sym}
after $n-1-i$ terms (counting from the left). In other words, the objects~$F_i$ are defined by the following exact sequences:
\begin{equation}
\label{eq-left-complexes-Fi}\begin{split}
0 \to S^{n-2}&\cU^\vee(2-n)  \to F_{n-2} \to 0 \\
0 \to S^{n-2}&\cU^\vee(2-n) \to W_1 \otimes S^{n-3}\cU^\vee(3-n) \to F_{n-3} \to 0 \\
  & \vdots \\
0 \to S^{n-2}&\cU^\vee(2-n) \to W_1 \otimes S^{n-3}\cU^\vee(3-n) \to \dots \to W_{n-3} \otimes \cU^\vee(-1) \to F_1 \to 0.
\end{split}
\end{equation}
Since~\eqref{eq-OG-2-2n-exact-sequence-for-Sym} is exact, the objects $F_i$ are also quasiisomorphic to the stupid left truncations
of~\eqref{eq-OG-2-2n-exact-sequence-for-Sym}; in other words, we also have exact sequences
\begin{equation}
\label{eq-right-complexes-Fi}
\begin{split}
0 \to &F_{n-2} \to  W_1 \otimes S^{n-3}\cU^\vee(3-n) \to  \dots \to W_1 \otimes S^{n-3} \cU^\vee(1) \to S^{n-2} \cU^\vee(1) \to 0 \\
0 \to &F_{n-3} \to  W_2 \otimes S^{n-4}\cU^\vee(4-n) \to  \dots \to W_1 \otimes S^{n-3} \cU^\vee(1) \to S^{n-2} \cU^\vee(1) \to 0 \\
& \ \vdots \\
0 \to &F_1 \to W_{n-2} \otimes \cO \to  \dots \to W_1 \otimes S^{n-3} \cU^\vee(1) \to S^{n-2} \cU^\vee(1) \to 0.
\end{split}
\end{equation}

We will check that the objects $F_1,\dots,F_{n-2}$ together with the objects $\cU^{2\omega_{n-1}}(-1)$, $\cU^{2\omega_{n}}(-1)$
form a full exceptional collection in the residual category.
We denote by~$\bL$ the left mutation functors.

\begin{lemma}
We have
\begin{align*}
\bL_{\langle \cA, \cA(1), \dots , \cA(i) \rangle} \big( S^{n-2} \cU^\vee(i) \big) \cong F_i (i-1) [n-1+i].
\end{align*}
\end{lemma}

\begin{proof}
To prove the claim it is enough to show the following two facts:
\begin{itemize}
\item
the object $F_i (i-1) [n-1+i]$ is contained in the orthogonal $\langle \cA, \cA(1), \dots , \cA(i) \rangle^\perp$;
\item
there exists a morphism $S^{n-2} \cU^\vee(i) \to F_i(i-1)[n-1+i]$,
whose cone lies in the category~\mbox{$\langle \cA, \cA(1), \dots , \cA(i) \rangle$}.
\end{itemize}

Twisting~\eqref{eq-left-complexes-Fi} by $\cO(i-1)$ and using semiorthogonality of~\eqref{eq-collection-OG-even-I},
we deduce the first of these facts.
On the other hand, considering~\eqref{eq-right-complexes-Fi} twisted by~$\cO(i-1)$
as a Yoneda extension of~$S^{n-2} \cU^\vee(i)$ by $F_i(i-1)$ of length $n-1+i$,
i.e., as a morphism $S^{n-2} \cU^\vee(i) \to F_i(i-1)[n-1+i]$,
we conclude that its cone is quasiisomorphic to the subcomplex of middle terms,
hence belongs to the subcategory $\langle \cA, \cA(1), \dots , \cA(i) \rangle$.
This proves the second fact.
\end{proof}

Thus, we have the following description for the residual category
\begin{equation}
\label{eq:residual-dn}
\cR = \Big\langle \cU^{2\omega_{n-1}}(-1), \cU^{2\omega_{n}}(-1) , F_1 , F_2(1), \dots , F_{n-2}(n-3) \Big\rangle.
\end{equation}

\begin{remark}
\label{remark:cr-dn-aut}
For $n \ne 4$ the exceptional collection in~\eqref{eq:residual-dn} is $\Aut(\OG(2,2n))$-invariant.
For~$n = 4$ it takes form $S^2\cS_-(-1), S^2\cS_+(-1), (V \otimes \cU)/S^2\cU, S^2\cU^\vee(-1)$.
Mutating the third object to the right we obtain an $\Aut(\OG(2,8))$-invariant exceptional collection
\begin{equation*}
\cR = \langle S^2\cS_-(-1), S^2\cS_+(-1), S^2\cU^\vee(-1), \tilde{T}(-1) \rangle,
\end{equation*}
where the bundle $\tilde T$ is defined by the exact sequence $0 \to \cO \to \tilde T \to T \to 0$ and $T$ is the tangent bundle.
\end{remark}

It remains to compute the $\Ext$-spaces between the objects of~\eqref{eq:residual-dn}.

\begin{lemma}
We have
\begin{equation*}
\Ext^\bullet(F_i(i-1), F_j(j-1)) =
\begin{cases}
\Bbbk,  		& \text{if $j = i + 1$ or $j = i$,}  \\
0 , 			& \text{otherwise.}
\end{cases}
\end{equation*}
Furthermore, we have
\begin{equation*}
\Ext^\bullet(\cU^{2\omega_{n-1}}(-1), F_j(j-1)) =
\Ext^\bullet(\cU^{2\omega_{n}}(-1), F_j(j-1)) =
\begin{cases}
\Bbbk[-1] ,  	& \text{if $j = 1$,} \\
0 , 		& \text{otherwise.}
\end{cases}
\end{equation*}
\end{lemma}

\begin{proof}
We start with the first claim.
Since we already know that~\eqref{eq:residual-dn} is an exceptional collection, we may assume $i < j$.
From~\eqref{eq-right-complexes-Fi} twisted by $\cO(i-1)$ we deduce
\begin{equation*}
F_i(i-1) \in \langle \cA , \cA(1) , \dots , \cA(i-1), \cB(i) \rangle
\end{equation*}
and from~\eqref{eq-left-complexes-Fi} twisted by $\cO(j-1)$ and with $i$ replaced by $j$ we deduce
\begin{equation}
\label{eq:fj-second}
F_j(j-1) \in \langle \cB(j-n+1) , \cA(j-n+2) , \dots , \cA(-1) \rangle.
\end{equation}
Using appropriate twist of \eqref{eq-collection-OG-even-I} one easily sees that non-trivial $\Ext$'s can only come from
\begin{equation*}
\Ext^\bullet(S^{n-2} \cU^\vee(i), S^{n-2} \cU^\vee(j-n+1)).
\end{equation*}
By Lemma~\ref{lemma:so1} this space is nonzero (and is isomorphic to $\Bbbk[4-2n]$)
only if $j = i + 1$ and the claim follows.

For the second claim, using the inclusion~\eqref{eq:fj-second}, we see that non-trivial $\Ext$'s can only come from
\begin{align*}
\Ext^\bullet(\cU^{2\omega_{\varepsilon}}(-1) , S^{n-2} \cU^\vee(i-n+1)),
\end{align*}
where $\varepsilon \in \{n-1, n\}$. By Lemma~\ref{lemma:so4} and Serre duality this space is nonzero (and is isomorphic to $\Bbbk[2-n]$) only if $i = 1$, and the claim follows.
\end{proof}

The above lemma shows that the residual category is equivalent to the derived category of representations
of a Dynkin quiver of type~$\rD_n$ and Remark~\ref{remark:cr-dn-aut} shows that it is generated by an exceptional
collection invariant under all automorphisms of~$\OG(2,2n)$.
This completes the proof of Theorem~\ref{theorem:dn}.


\begin{bibdiv}
\begin{biblist}

\bib{BS}{article}{
      author={Belmans, P.},
      author={Smirnov, M.},
       title={The {H}ochschild cohomology of generalised {G}rassmannians},
      eprint={https://arxiv.org/abs/1911.09414},
}

\bib{BKS}{article}{
      author={Belmans, Pieter},
      author={Kuznetsov, Alexander},
      author={Smirnov, Maxim},
       title={Derived categories of the {C}ayley plane and the coadjoint
  {G}rassmannian of type~$\mathrm{F}$},
      eprint={https://arxiv.org/abs/2005.01989},
}

\bib{Bott}{article}{
      author={Bott, Raoul},
       title={Homogeneous vector bundles},
        date={1957},
     journal={Ann. of Math. (2)},
      volume={66},
       pages={203\ndash 248},
}

\bib{BKT09}{article}{
      author={Buch, Anders~Skovsted},
      author={Kresch, Andrew},
      author={Tamvakis, Harry},
       title={Quantum {P}ieri rules for isotropic {G}rassmannians},
        date={2009},
     journal={Invent. Math.},
      volume={178},
      number={2},
       pages={345\ndash 405},
}

\bib{ChPe}{article}{
      author={Chaput, P.~E.},
      author={Perrin, N.},
       title={On the quantum cohomology of adjoint varieties},
        date={2011},
     journal={Proc. Lond. Math. Soc. (3)},
      volume={103},
      number={2},
       pages={294\ndash 330},
}

\bib{CF99}{article}{
      author={Ciocan-Fontanine, Ionu\c{t}},
       title={On quantum cohomology rings of partial flag varieties},
        date={1999},
     journal={Duke Math. J.},
      volume={98},
      number={3},
       pages={485\ndash 524},
}

\bib{CMKMPS}{article}{
      author={Cruz~Morales, J.A.},
      author={Kuznetsov, A.},
      author={Mellit, A.},
      author={Perrin, N.},
      author={Smirnov, M.},
       title={On quantum cohomology of {Grassmannian}s of isotropic lines,
  unfoldings of {$A_n$}-singularities, and {Lefschetz} exceptional
  collections},
        date={2019},
     journal={Annales de l'Institut Fourier},
      volume={69},
      number={3},
       pages={955\ndash 991},
}

\bib{Du}{inproceedings}{
      author={Dubrovin, Boris},
       title={Geometry and analytic theory of {F}robenius manifolds},
        date={1998},
   booktitle={Proceedings of the {I}nternational {C}ongress of
  {M}athematicians, {V}ol. {II} ({B}erlin, 1998)},
       pages={315\ndash 326},
}

\bib{FP}{incollection}{
      author={Fulton, W.},
      author={Pandharipande, R.},
       title={Notes on stable maps and quantum cohomology},
        date={1997},
   booktitle={Algebraic geometry---{S}anta {C}ruz 1995},
      series={Proc. Sympos. Pure Math.},
      volume={62},
   publisher={Amer. Math. Soc., Providence, RI},
       pages={45\ndash 96},
}

\bib{GMS}{article}{
      author={Galkin, S.},
      author={Mellit, A.},
      author={Smirnov, M.},
       title={Dubrovin's conjecture for {$IG(2,6)$}},
        date={2015},
     journal={Int. Math. Res. Not.},
      volume={2015},
      number={18},
       pages={8847\ndash 8859},
}

\bib{HeMaTe}{article}{
      author={Hertling, C.},
      author={Manin, Yu.~I.},
      author={Teleman, C.},
       title={An update on semisimple quantum cohomology and {$F$}-manifolds},
        date={2009},
        ISSN={0371-9685},
     journal={Trudy Matematicheskogo Instituta Imeni V. A. Steklova. Rossi\u\i
  skaya Akademiya Nauk},
      volume={264},
      number={Mnogomernaya Algebraicheskaya Geometriya},
       pages={69\ndash 76},
      review={\MR{2590836}},
}

\bib{Hertling}{book}{
      author={Hertling, Claus},
       title={Frobenius manifolds and moduli spaces for singularities},
      series={Cambridge Tracts in Mathematics},
   publisher={Cambridge University Press, Cambridge},
        date={2002},
      volume={151},
}

\bib{IK}{article}{
      author={Ingalls, Colin},
      author={Kuznetsov, Alexander},
       title={On nodal {E}nriques surfaces and quartic double solids},
        date={2015},
        ISSN={0025-5831},
     journal={Math. Ann.},
      volume={361},
      number={1-2},
       pages={107\ndash 133},
         url={https://doi.org/10.1007/s00208-014-1066-y},
      review={\MR{3302614}},
}

\bib{Kaufmann}{article}{
      author={Kaufmann, Ralph},
       title={The intersection form in {${H}^*(\overline{M}_{0n})$} and the
  explicit {K}\"{u}nneth formula in quantum cohomology},
        date={1996},
     journal={Internat. Math. Res. Notices},
      number={19},
       pages={929\ndash 952},
}

\bib{Ke}{article}{
      author={Ke, Hua-Zhong},
       title={On semisimplicity of quantum cohomology of
  {$\mathbb{P}^1$}-orbifolds},
        date={2019},
     journal={J. Geom. Phys.},
      volume={144},
       pages={1\ndash 14},
}

\bib{Kim}{article}{
      author={Kim, Bumsig},
       title={Quantum cohomology of partial flag manifolds and a residue
  formula for their intersection pairings},
        date={1995},
     journal={Internat. Math. Res. Notices},
      number={1},
       pages={1\ndash 15},
}

\bib{KMK}{article}{
      author={Kontsevich, M.},
      author={Manin, Yu.},
       title={Quantum cohomology of a product},
        date={1996},
     journal={Invent. Math.},
      volume={124},
      number={1-3},
       pages={313\ndash 339},
        note={With an appendix by R. Kaufmann},
}

\bib{K06}{article}{
      author={Kuznetsov, A.~G.},
       title={Hyperplane sections and derived categories},
        date={2006},
     journal={Izv. Ross. Akad. Nauk Ser. Mat.},
      volume={70},
      number={3},
       pages={23\ndash 128},
}

\bib{Ku08a}{article}{
      author={Kuznetsov, Alexander},
       title={Exceptional collections for {G}rassmannians of isotropic lines},
        date={2008},
     journal={Proceedings of the London Mathematical Society. Third Series},
      volume={97},
      number={1},
       pages={155\ndash 182},
}

\bib{KuPo}{article}{
      author={Kuznetsov, Alexander},
      author={Polishchuk, Alexander},
       title={Exceptional collections on isotropic {G}rassmannians},
        date={2016},
     journal={J. Eur. Math. Soc. (JEMS)},
      volume={18},
      number={3},
       pages={507\ndash 574},
}

\bib{KS2020}{article}{
      author={Kuznetsov, Alexander},
      author={Smirnov, Maxim},
       title={On residual categories for {G}rassmannians},
     journal={Proceedings of the London Mathematical Society},
      volume={120},
      number={5},
       pages={617\ndash 641},
}

\bib{Manin}{book}{
      author={Manin, Yuri~I.},
       title={Frobenius manifolds, quantum cohomology, and moduli spaces},
      series={American Mathematical Society Colloquium Publications},
   publisher={American Mathematical Society, Providence, RI},
        date={1999},
      volume={47},
}

\bib{Mironov}{article}{
      author={Mironov, Mikhail},
       title={Lefschetz exceptional collections in {$S_k$}-equivariant
  categories of $(\mathbb{P}^n)^k$},
      eprint={https://arxiv.org/abs/1807.01534},
}

\bib{PeSm}{article}{
      author={Perrin, N.},
      author={Smirnov, M.},
       title={On the big quantum cohomology of (co)adjoint varieties},
     journal={In preparation},
}

\bib{Pe}{article}{
      author={Perrin, Nicolas},
       title={Semisimple quantum cohomology of some {F}ano varieties},
      eprint={https://arxiv.org/abs/1405.5914},
}

\bib{Po}{article}{
      author={Polishchuk, A.},
       title={{$K$}-theoretic exceptional collections at roots of unity},
        date={2011},
        ISSN={1865-2433},
     journal={J. K-Theory},
      volume={7},
      number={1},
       pages={169\ndash 201},
}

\bib{VO}{book}{
      author={Vinberg, \`E.~B.},
      author={Onishchik, A.~L.},
       title={{Seminar po gruppam Li i algebraicheskim gruppam}},
     edition={Second},
   publisher={URSS, Moscow},
        date={1995},
}

\end{biblist}
\end{bibdiv}
\end{document}